\documentclass[12pt]{amsart}

\usepackage{graphicx}
\usepackage{amsmath, comment}
\usepackage{amscd}
\usepackage{amsfonts}
\usepackage{amssymb}
\usepackage{color}
\usepackage{fullpage}
\usepackage{mathtools}
\usepackage[hypertexnames=false, colorlinks, citecolor=red,linkcolor=blue, urlcolor=black]{hyperref}
\usepackage{tikz-cd}
\usepackage[capitalise]{cleveref}
\usepackage{enumitem}
\usepackage{booktabs}

\numberwithin{equation}{section}

\newtheorem*{theorem*}{Theorem}
\newtheorem{theorem}{Theorem}[section]
\newtheorem{lemma}[theorem]{Lemma}

\newtheorem{proposition}[theorem]{Proposition}

\newtheorem{corollary}[theorem]{Corollary}

\theoremstyle{definition}
\newtheorem{example}[theorem]{Example}

\newtheorem{remark}[theorem]{Remark}
\newtheorem{definition}[theorem]{Definition}

\newcommand{\cF}{\mathcal{F}}

\newcommand{\cO}{\mathcal{O}}
\newcommand{\cP}{\mathcal{P}}

\newcommand{\bF}{\mathbb{F}}

\newcommand{\bN}{\mathbb{N}}
\newcommand{\bQ}{\mathbb{Q}}

\newcommand{\bT}{\mathbb{T}}

\newcommand{\bZ}{\mathbb{Z}}

\newcommand{\id}{\operatorname{id}}

\newcommand{\im}{\operatorname{im}}

\newcommand{\Gal}{\operatorname{Gal}}

\newcommand{\PGL}{\operatorname{PGL}}

\newcommand{\Hom}{\operatorname{Hom}}

\newcommand{\Br}{\operatorname{Br}}

\newcommand{\SL}{\operatorname{SL}}
\newcommand{\SO}{\operatorname{SO}}
\newcommand{\Spin}{\operatorname{Spin}}

\newcommand{\Spec}{\operatorname{Spec}}

\newcommand{\trdeg}{\operatorname{trdeg}}

\newcommand{\ed}{\operatorname{ed}}

\newcommand{\rank}{\operatorname{rank}}

\newcommand{\sep}{\operatorname{sep}}
\newcommand{\sat}{\operatorname{sat}}
\newcommand{\colim}{\operatorname{colim}}

\newcommand{\Char}{\operatorname{char}}

\newcommand{\Sets}{\operatorname{Sets}}

\newcommand{\Mat}{\operatorname{M}}

\newcommand{\CSA}{\operatorname{CSA}}

\newcommand{\Fields}{\operatorname{Fields}}
\newcommand{\Set}{\operatorname{Sets}}

\begin{document}
\author{Danny Ofek}   
\thanks{Danny Ofek was partially supported by a graduate fellowship from 
the University of British Columbia.} 

\author{Zinovy Reichstein}
\address{Department of Mathematics\\
	% 1984 Mathematics Road\\
	University of British Columbia\\
	Vancouver, BC V6T 1Z2\\Canada}
% \email{reichst@math.ubc.ca}
\thanks{Zinovy Reichstein was partially supported by an Individual Discovery Grant from the
	National Sciences and Engineering Research Council of
	Canada.}

\subjclass[2020]{12G05, 12J20, 16K50, 19C30}

%%%%%%%%%%%%%%%%%%%%%%%%%%%%%%%%%%%%%%%%
% 12 (1959-now) Field theory and polynomials
% 12G (1973-now) Homological methods (field theory)
% 12G05 (1973-now) Galois cohomology [See also 14F22, 16H05, 16K50]
% 
% 12 (1959-now) Field theory and polynomials
% 12J (1973-now) Topological fields
% 12J20 (1973-now) General valuation theory for fields [See also 13A18]
%
% 16 (1959-now) Associative rings and algebras {For the commutative case, see 13-XX}
% 16K (1991-now) Division rings and semisimple Artin rings [See also 12E15, 15A30]
% 16K50 (2000-now) Brauer groups (algebraic aspects) [See also 12G05, 14F22]
%
% 19 (1986-now) K-theory [See also 16E20, 18F25]
% 19C (1986-now) Steinberg groups and K2
% 19C30 (1986-now) K2 and the Brauer group
%%%%%%%%%%%%%%%%%%%%%%%%%%%%%%%%%%%%%%%%
\keywords{Essential dimension, Galois cohomology, Brauer group, valued field}

\title{Essential dimension of cohomology classes via valuation theory}
\maketitle

\begin{abstract} We give a formula for the essential dimension of a cohomology class $\alpha$ in $H^d(K, \bQ_p/\bZ_p (d))$ when
$K$ is a strictly Henselian field. This formula is particularly explicit in the case, where $\alpha$ is 
a Brauer class (for $d = 2$). As an application of our bound with $d = 3$, we study the essential dimension of exceptional 
groups by examining the image of the Rost invariant.
\end{abstract}

\section{Introduction}

% The essential dimension of an algebraic object is the minimal number of independent parameters required to define it.

Let $k$ be a base field and $k \subset K$ be a field extension. Consider an algebraic object  
$\alpha$ defined over $K$. That is, $\alpha \in \cF(K)$ is an object of a covariant functor $\cF : \Fields_k \to \Set$,
where $\Fields_k$ is the category of field extensions $K/k$ and $\Sets$ is the category of sets. 
We think of the functor $\mathcal{F}$ as specifying the type of object under consideration
(a quadratic form, an associative algebra, a Lie algebra, etc.) and
$\mathcal{F}(K)$ as the set of isomorphism classes of objects this type defined over $K$. 
The essential dimension $\ed_k(\alpha)$ of our object $\alpha \in \mathcal{F}(K)$ is the minimal transcendence degree
$\trdeg_k(K_0)$ of an intermediate field $k \subset K_0 \subset K$ to which $\alpha$ descends.
Informally speaking, $\ed_k(\alpha)$ is the minimal number of parameters required to define $\alpha$.
The essential dimension $\ed_k(\cF)$ of the functor $\mathcal{F}$ is then defined as the maximal value 
of $\ed_k(\alpha)$ taken over all $L \in \Fields_k$ and all
$\alpha \in \cF(L)$. These definitions are recalled in a more formal way in Section~\ref{sect.preliminaries}. 
For surveys of this research area, we refer the reader to~\cite{berhuy2003essential, merkurjev-survey, reichstein2010essential}.

Much of the work on essential dimension has centered on the functor
\begin{equation} \label{e.h1} H^1(\ast, G) \colon K \mapsto H^1(K, G). 
\end{equation}
Here $H^1(K, G)$ denotes the non-abelian cohomology set, whose elements are isomorphism classes of (fppf) $G$-torsors 
over $\Spec(K)$. The essential dimension of this functor is called the essential dimension of $G$ and is denoted 
by $\ed_k(G)$. 

The main focus of this paper will be on a different functor,
$$H_p^d \colon K \mapsto H^d(K, \bQ_p/\bZ_p (d)), $$
where $p \neq \Char k$ is a prime.  We will study the essential dimension of objects of this functor (i.e., cohomology classes)
by valuation-theoretic methods, in the spirit of~\cite[Section 3f]{merkurjev2009essential} or \cite{meyer2012valuation}.
%%%%%%%%%%%%%%%%%%%%%%%%%%%%%%%%%%%%%%%%%
% Using our results on the essential dimension of $H^d_p$, we will prove new and old lower bounds on the essential dimension of torsors. We will % now sketch our main result about the essential dimension of $H_p^d$, and in the next two subsections we explain the applications. 
%%%%%%%%%%%%%%%%%%%%%%%%%%%%%%%%%%%%%%%%%

Let $(F,\nu)$ be a valued field with value group $\bZ^r$. Assume $k\subset F$ and $\nu$ is trivial on $k$. 
There is a canonical homomorphism 
$$\wedge\nu: H^d_p(F) \longrightarrow \bQ_p/\bZ_p \otimes \bigwedge \bZ^r, $$
where $\bigwedge \bZ^r$ is the exterior algebra on $\bZ^r$. It is given, on symbols, by
$$ \wedge\nu (a_1,\dots,a_d)_{p^n} = \frac{1}{p^n}\otimes \nu(a_1)\wedge\dots\wedge \nu(a_d); $$
for details see Section~\ref{sect.construction_of_nu}.

In Definition~\ref{def.contraction}, we associate to any $\omega \in  \bQ_p/\bZ_p \otimes \bigwedge \bZ^r$,
a finite abelian subgroup $A_\omega \subset (\bQ_p/\bZ_p)^r$. Given $\omega$, the group $A_{\omega}$ is usually 
easy to compute. Our main result is the following:

\begin{theorem}\label{thm.main} Let $(F,\nu)$ be a valued field with value group $\bZ^r$. Assume $\nu$ is trivial on a subfield $k\subset F$ with $\Char k \neq p$. Let $\alpha \in H^d_p(F)$ and $\omega = \wedge\nu(\alpha)$. Then  
     $$ \ed_k(\alpha) \geqslant \ed_k(\alpha; p) \geqslant  \dim_{\bF_p}(A_\omega/ pA_\omega).$$
    Moreover, if $(F,\nu)$ is strictly Henselian, then equality holds:
    $$\ed_k(\alpha) = \ed_k(\alpha; p) = \dim_{\bF_p}(A_\omega/ pA_\omega).$$
\end{theorem}
Here $\ed_k(\alpha; p)$ is \emph{the essential dimension of $\alpha$ at $p$}; see Section~\ref{sect.prel1}.
Note that first inequality in the statement of Theorem~\ref{thm.main}, $\ed_k(\alpha) \geqslant \ed_k(\alpha; p)$,
is immediate from the definition of $\ed_k(\alpha; p)$; see~\eqref{e.ed-vs-ed-at-p}. 
As a consequence of Theorem~\ref{thm.main} we obtain the following: 

\begin{corollary} \label{cor.infty} Let $p$ be a prime and $k$ be a field of characteristic $\neq p$. Then  
\begin{equation}\label{e.infty}
    \ed_k(H^d_p; p) = \infty
\end{equation}
for any $d \geqslant 2$. 
\end{corollary}

For a proof, see Section~\ref{sect.main}. 

\subsection*{Central simple algebras and Brauer classes}
Let us now consider the functor~\eqref{e.h1} in the special case, where $G$ is the projective linear group $\PGL_n$.
It is well known that the functor $H^1(*, \PGL_n)$ is isomorphic to
$$\CSA_n \colon K \mapsto \{\text{isomorphism classes of central simple algebras of degree $n$ over $K$} \}.$$
Computing the essential dimension of $\CSA_n$ is a long-standing problem, going back to Procesi~\cite[Section 2]{procesi}.
The deepest results to date have been about $\ed_k(\CSA_n; p)$. Using primary decomposition
it is easy to see that $\ed(\CSA_n; p) = \ed_k(\CSA_{p^r}; p)$, where $p^r$ is the highest power of $p$ dividing $n$. 
Thus we may assume without loss of generality that $n = p^r$.
When $r = 1$, it is known that $\ed(\PGL_p; 2) = 2$; see, 
e.g., \cite[Lemma 8.5.7]{reichstein2000essential}.
For $r \geqslant 2$, we have
\[ (r-1)p^r + 1 \leqslant \ed(\PGL_{p^r}; p) \leqslant p^{2r-2} + 1 \, , \]
where the lower bound is due to Merkurjev~\cite{merkurjev-PGLn} and the upper bound
is due to Ruozzi~\cite{ruozzi}. In particular,
\begin{equation} \label{e.PGL_p^2}
\ed_k(\PGL_{p^2}) = \ed_k(\PGL_{p^2}; p) = p^2 + 1; 
\end{equation}
see~\cite{merkurjev-PGLp^2}.

Now consider the morphism of functors $\CSA_n \to \Br$ taking a central simple algebra $A$ to its Brauer class $[A]$.
If $A$ is not split and $k$ is algebraically closed, then Tsen's Theorem tells us that
\begin{equation} \label{e.brauer1}
\ed_k([A]; p) \geqslant 2.
\end{equation}
We also have the obvious inequalities 
\begin{equation} \label{e.brauer2}
\ed_k([A]) \leqslant \ed_k(A) \quad \text{and} \ed_k([A]; p) \leqslant \ed_k(A; p) \, . 
\end{equation}
The inequalities~\eqref{e.brauer2} may be sharp, because 
the essential dimension of $\Mat_s(A)$ may be strictly smaller than the essential dimension of $A$ for some $s > 1$. 
In fact, this is exactly what happens when $A$ is a universal division algebra of degree $4$. In this case~\eqref{e.PGL_p^2} 
tells us that $\ed_k(A) = \ed_k(A; 2) = 5$; on the other hand, by \cite[Corollary 1.4]{lorenz2003}, $\ed_k(\Mat_2(A)) = 4$. 

Using Theorem~\ref{thm.main}, we can give lower bounds on the essential dimension of some Brauer classes that go beyond~\eqref{e.brauer1}. 

\begin{theorem}\label{thm.brauer} Let $p$ be a prime, $k$ be a field containing a primitive root of unity of degree $p^d$ for every $d \geqslant 1$.
Let $(F,\nu)$ be a valued field with value group $\bZ^r$. Assume that $k\subset F$ and $\nu_{|k}$ is trivial. Let $\alpha \in \Br(F)$ is a sum of Brauer classes of symbol algebras:
$$\alpha = (a_1,b_1)_{p^n} + \dots + (a_r,b_r)_{p^n},$$
for some integer $n \geqslant 1$. Consider the skew-symmetric matrix 
$$M = \sum_{i=1,\dots,k} \nu(a_i)\nu(b_i)^t - \nu(b_i)\nu(a_i)^t \in \Mat_r(\bZ), $$
where we view $\nu(a_i)$ and $\nu(b_i) \in \bZ^r$ as $r \times 1$ matrices with integer entries
and their transposes, $\nu(a_i)^t$ and $\nu(b_i)^t$ as $1 \times r$ matrices.
Let $d_1 \mid d_2 \mid \dots \mid d_r$ be the elementary divisors of $M$ and assume $i_0$ is the largest subscript 
such that $p^n$ does not divide $d_{i_0}$. Then 

\smallskip
(a) $\ed_k(\alpha; p) \geqslant i_0$, where $\alpha$ is viewed as an object of the functor $\Br$.
In particular, if $p^r$ does not divide $\det(M)$, then $\ed_k(\alpha)\geqslant r$. 

\smallskip
(b) If $(F,\nu)$ is strictly Henselian, then $\ed_k(\alpha) = \ed_k(\alpha ;p) = i_0$.
\end{theorem}

The assumption that $\Char k \neq p$ is crucial here. 
If $\Char(k) = p$, then ~\eqref{e.brauer1} is tight; see Proposition~\ref{prop.char-p}. 

As far as we know, Theorem~\ref{thm.brauer} is the first known bound on $\ed_k(\alpha)$ that is stronger than \eqref{e.brauer1}.
There are some lower bounds in the literature on the essential dimension of Brauer classes considered as objects of $\Br_p$, where $\Br_p(K)$ denotes the $p$-torsion subgroup of $\Br(K)$ (see \cite{mckinnie2017essential} and the last paragraph of Section 2 in~\cite{reichstein2010essential}).
When a Brauer class $\alpha$ of exponent $p$ is considered as an object of $\Br_p$, it is only allowed to descend to Brauer classes of exponent $p$. 
When we view $\alpha$ as an object of $\Br$ (as we do in the setting of Theorem~\ref{thm.brauer}), 
it is allowed to descend to a Brauer class of higher exponent, so a priori the essential dimension may drop.

\subsection*{Essential dimension of exceptional groups and the Rost invariant}
% %%%%%%%%%%%%%%%%%%%%%%%%%%%%%%%%%%%%%%%%%
A cohomological invariant of an algebraic group $G$ is a morphism of functors:
$$\eta: H^1(\ast, G) \to H^d(\ast,M).$$
Here $M$ is any discrete $\Gal(k)$-module.  We refer the reader to \cite{serre-ci} 
for a detailed discussion of cohomological invariants. 
We will assume that $\eta$ is normalized in the sense of~\cite[Chapter I, 4.5]{serre-ci}, i.e., takes the trivial $G$-torsor to zero. 

As an easy consequence of the definition of essential dimension, we see that
\begin{equation}\label{eq.lowerbound_coh_invariant}
    \ed(G;p)\geqslant  \ed(\gamma; p) \geqslant  \ed(\eta(\gamma);p);
\end{equation}
cf.~\cite[Lemma 2.2]{reichstein2010essential}. Of particular interest to us will be the Rost invariant
\[ R_G \colon H^1(\ast, G) \to H^3(\ast,\bQ_p/\bZ_p(2)) . \]
Recall that the Rost invariant is defined for every semisimple simply connected group $G$; 
see~\cite[Section 31B]{book_involution} or~\cite{merkurjev-ci}.

%%%%%%%%%%%%%%%%%%%%%%%%%%%%%%%%%%%%%%%%%%%%%%%%%%%%%%%%%%%
% Assume that $\eta$ is valued in $H^d(\ast,\bQ_p/\bZ_p)$ for some $p\neq \Char k$ and $k=k^{\sep}$. Serre vanishing implies that any non-trivial % element $\alpha \in H^d(F,\bQ_p/\bZ_p)$ satisfies $\ed(\alpha;p) \geqslant  d$. Therefore if $\eta$ is non-trivial, then 
%  \ref{eq.lowerbound_coh_invariant} implies:
%  \begin{equation}\label{eq.ed_coh_inv_weakbound}
%   \ed(G;p)\geqslant  d.
%  \end{equation}
% This weaker inequality has been used to give lower bounds for many algebraic groups. For example, all known lower bounds on $\ed(G;p)$ can be 
% derived from \eqref{eq.ed_coh_inv_weakbound} if $G$ is of type $G_2,F_4$ or $E^{sc}_6$. For groups of type $E_7$ or $E_8$ there are no known 
% cohomological invariants of degree greater than three and so \eqref{eq.ed_coh_inv_weakbound} cannot be used to prove the following inequalities.
%%%%%%%%%%%%%%%%%%%%%%%%%%%%%%%%%%%%%%%%%%%%%%%%%%%%%%%%%%%%%%%%
In Section~\ref{sect.rost} we will use the inequality~\eqref{eq.lowerbound_coh_invariant} in combination with Theorem~\ref{thm.main} to prove the following
lower bounds:
\begin{equation} \label{e.exceptional}
\text{(i)  \;} \ed_k(E_7^{\rm sc};2) \geqslant  7, \quad \quad \text{(ii) \;} \ed_k(E_8;2)\geqslant  9, \quad \quad \text{(iii) \;} \ed_k(E_8; 3)\geqslant  5.
\end{equation}
Here the base field $k$ is assumed to be of characteristic different from $2$ in parts (i) and (ii) and different from $3$ in (iii).
Specifically, we will set $d = 3$, $\eta = R_G$ to be the Rost invariant. The field $F$ will be $F_7$, $F_9$, and $F_5$ in parts (i), (ii) and (iii), 
respectively, where $F = F_n$ is the iterated power series field $F_n = k((t_0))((t_2)) \ldots ((t_{n-1}))$ 
\footnote{In part (iii), we will first adjoin a primitive 3rd root of unity to $k$. 
This is harmless, as enlarging $k$ does not increase the essential dimension.}.

% In (i) and (ii) we assume that $\Char(k) \neq 2$. In (iii) we assume that $k$ contains a primitive $3$-rd root of unity.

Note that the inequalities~\eqref{e.exceptional} are known. In characteristic $0$ they were first proved in~\cite{reichstein-youssin}. 
In full generality (i) and (ii) were proved in~\cite{chernousov2006lower} and (iii) in~\cite{gille2009lower}. 
Moreover, the arguments used in \cite{chernousov2006lower} and \cite{gille2009lower} show
that there exists an $E_7^{\rm sc}$-torsor $T \to \Spec(F_7)$ such that $\ed_k(T; 2) = 7$, and similarly in parts (ii) and (iii). 
What is new here is that the lower bounds in~\eqref{e.exceptional} can, in fact, be extracted from the Rost invariant of $T$ 
in the following sense.

\begin{theorem} \label{thm.rost} Let $k$ be a base field of characteristic $\neq 2$ and 
$F_n$ be the iterated Laurent series field $F_n = k((t_0))((t_2)) \ldots ((t_{n-1}))$. Then there exist
(i) an $E_7^{\rm sc}$-torsor $T_1 \to \Spec(F_7)$, (ii) an $E_8$-torsor $T_2 \to \Spec(F_9)$, and (iii) an $E_8$-torsor $T_3 \to \Spec(F_5)$
such that (i) $\ed_k(R_{E_7^{\rm sc}}(T_1); 2) = 7$, (ii) $\ed_k(R_{E_8}(T_2); 2) = 9$, and (iii) $\ed_k(R_{E_8}(T_3); 3) = 5$, respectively.

Here in (iii) we are assuming that $k$ contains a primitive $3$rd root of unity. 
\end{theorem}

Our proof of Theorem~\ref{thm.rost} in Section~\ref{sect.rost} relies on the formulas for the Rost invariant, due to  Chernousov~\cite{chernousov2003kernel} and Garibaldi~\cite{garibaldi2009orthogonal}. We also note that the exponential lower bounds on $\ed_k(\Spin_n; 2)$ from~\cite{brv-annals} cannot be recovered by this method; see Remark~\ref{rem.spin}.

\section{Notation and preliminaries}
\label{sect.preliminaries}

\subsection{Essential dimension} 
\label{sect.prel1}
Let $k$ be a base field, $\Fields_k$ be the category of field extensions $K/k$, $\Sets$ be the category of sets, and $\mathcal{F} \colon \Fields_k \to \Set$ be a covariant functor.
Given a field extension $K/k$, we will say that $\alpha\in \mathcal{F}(K)$ \emph{descends} to an intermediate field $k \subseteq K_0 \subseteq K$ if $\alpha$ is in the image of the induced map $\mathcal{F}(K_0) \to \mathcal{F}(K)$.
The essential dimension $\ed_k(\alpha)$ is the smallest integer $d$ such that $\alpha$ descends to a field $k\subset K_0\subset K$ with $\trdeg_k(K_0) = d$.
The essential dimension $\ed_k(\alpha; p)$ of $\alpha$ at a prime $p$ is defined as the minimal value of $\ed_k(\alpha_L)$,
where $L/K$ ranges over field extensions whose degree $[L:K]$ is finite and prime to $p$.

The essential dimension $\ed_k(\mathcal{F})$ (respectively, the essential $p$-dimension $\ed_k(\mathcal{F}; p)$) 
of the functor $\mathcal{F}$ is the supremum of $\ed_k(\alpha)$ (respectively, $\ed_k(\alpha; p)$) taken over all $\alpha\in \mathcal{F}(K)$ with $K$ in $\Fields_k$. Clearly
\begin{equation} \label{e.ed-vs-ed-at-p}
\ed_k(\alpha) \geqslant \ed_k(\alpha; p) \quad \text{and} \quad \ed_k(\mathcal{F}) \geqslant \ed_k(\mathcal{F}; p) 
\end{equation}
for every object $\alpha$ of $\mathcal{F}$ and every prime $p$.

We also remark that if $f \colon \mathcal{F} \to \mathcal{G}$ is a morphism of functors from $\Fields_k$ to $\Sets$, then
\begin{equation} \label{e.morphism} \ed_k(\alpha) \geqslant \ed_k(f(\alpha)) \quad \text{and} \quad \ed_k(\alpha; p) \geqslant 
\ed_k(f(\alpha); p)
\end{equation}
for every object $\alpha$ of $\mathcal{F}$; cf.~\cite[Lemma 2.2]{reichstein2010essential}.
If $k \subset k' \subset L \subset L'$ are fields with $k'/k$ algebraic and $\alpha \in \cF(L)$, then
\begin{equation} \label{e.ext_of_scalar} 
\ed_k(\alpha; p) \leqslant \ed_{k}(\alpha_{L'};p) \quad \text{and} \quad \ed_{k}(\alpha;p)= \ed_{k'}(\alpha;p).
\end{equation}
Here the inequality on the left follows from the fact that any prime-to-$p$ extension of $L$ embeds into a prime-to-$p$ extension of $L'$, see \cite[Lemma 6.1]{merkurjev2009essential}. The equality on the right follows from the fact that
$\trdeg_k(L_0) = \trdeg_{k'}(L_0)$ for any intermediate field $k \subset k' \subset L_0 \subset L'$.

\subsection{The norm residue isomorphism} 
\label{sect.prel2}

Let $F$ be a field over $k$ and let $p$ be a prime different from $\Char k$. For any integer $d\in \bN$ we set 
$$\bQ_p/\bZ_p(d) := \underset{r\in \bN}{\colim} \ \bZ/p^r (d);$$
see \cite[Definition 7.3.6]{neukirch2013cohomology}.
Here the colimit is taken relative to the maps $\bZ/p^r \to \bZ/p^{s}$ given by multiplication by ${p^{s-r}}$, where $s \geqslant r \geqslant 0$. Note that $\bQ_p/\bZ_p(1)$ is isomorphic to the $p$-primary part of the group of roots of unity in $k^{\sep}$. 
Set 
% The following will be called the $p$-primary part of the cohomology of $F$:
$$H_p(F) := \bigoplus_{d\in \bN} H^d(F, \bQ_p/\bZ_p(d)).$$
This abelian group is naturally graded and functorial in $F$. For any $a_1,\dots,a_d\in F^*$ and $n\in\bN$, the Kummer map gives cohomology classes $(a_i)_{p^{n}} \in H^1(F,\mu_{p^n}) \cong F^*/F^{*p^n}$. The cup product of these classes defines a class $(a_1,\dots,a_d)_{p^n} \in H^d_p(F)$ which is called a symbol of degree $d$. The norm residue isomorphism theorem gives a simple presentation for $H_p(F)$ in terms of symbols.

\begin{theorem}\label{thm.norm}
Let $K(F)$ denote the Milnor K-theory of a field $F$. For any prime $p$ different from $\Char F$, there is an isomorphism of abelian graded groups $\bQ_p/\bZ_p\otimes K(F) \overset{h}{\to} H_p(F)$ given on generators by:
$$h(\frac{1}{p^n}\otimes \{a_1,\dots,a_d\}) = (a_1,\dots,a_n)_{p^n}$$
\end{theorem}
Theorem~\ref{thm.norm} is equivalent to the standard formulation of the norm residue isomorphism theorem. We include an explanation of the equivalence in the appendix.

\begin{remark}
    The norm residue isomorphism theorem is notoriously difficult to prove for general fields. For strictly Henselian fields it is much simpler. In this setting it is an easy consequence of~\cite[Theorem 2.6]{elman1982arason}; see \cite[Corollary 3.13]{wadsworth1983henselian}. To simplify the exposition, we will appeal to the norm residue theorem over general fields. Some readers may prefer to pass to the strict Henselization of a field before applying the norm residue isomorphism theorem. Doing this repeatedly will show that the proofs of our main results only require the norm residue isomorphism theorem over strictly Henselian fields.
\end{remark}

\subsection{Notational conventions}
\label{sect.notational-conventions}
Throughout this paper $k$ will denote a base field and $k^{\rm sep}$ will denote the separable closure of $k$.

Let $A$ be an abelian group. For any positive integer $m$, we set $A/m := A/mA$. 
We will write elements of $A/m$ as $a \mod m$, where $a\in A$. 
If $p$ is a prime, we will write 
$$A/p^\infty := \bQ_p/\bZ_p \otimes_\bZ A = \colim_{n\in \bN} A/p^n.$$

We denote the $n$-th exterior product of $A$ by $\bigwedge^n A$. When $A$ is a finitely-generated free $\bZ$ or $\bZ/m$-module and $e_1, \ldots, e_d$ is a basis of $A$, $\bigwedge^n A$ is also a free module generated by the ``pure wedges" $e_{i_1} \wedge \ldots \wedge e_{i_n}$ for some 
$1 < i_1 < \ldots < i_n \leqslant d$.

Note that there is a canonical isomorphism between $(\bigwedge^n A)/m$ and $\bigwedge^n (A/m)$. We will identify these groups and simply write $\bigwedge^n A/m$. 
The group $\bigwedge A/p^\infty$ is $(\bigwedge A)/p^\infty$ and not $\bigwedge (A/p^\infty)$ (the latter is the trivial group).

We fix a compatible system of roots of unity $\zeta_{m} \in k^{\sep}$ for all $m$ different from $\Char(k)$.
That is, $\zeta_{m_1 m_2}^{m_1} = \zeta_{m_2}$.
% In other words, we assume $\zeta^p_{p^{n+1}} = \zeta_{p^n}$ for any $n$ . 
This is equivalent to fixing compatible 
isomorphisms $\mu_{m}\to \bZ/m(1)$, where $\bZ/m(1)$ is the Tate twist of $\bZ/m$; see 
\cite[Definition 7.3.6]{neukirch2013cohomology}. 

A \emph{valued field over} $k$ is a field $F$ equipped with a valuation
$\nu \colon F^* \to \mathbb Z^r$ such that $k\subset F$, $\nu(k^*) = 0$. We will call $(F,\nu)$ \emph{strictly Henselian} if it satisfies Hensel's Lemma and its residue field is separably closed. We will need to consider the \emph{Henselization} and the \emph{strict Henselization} of a valued field. We refer the reader to \cite[Appendix A.3]{tignol2015value} for the definition and properties of the Henselization. The strict Henselization of a valued field is the \emph{inertial closure} of its Henselization (also called its maximal unramified extension) \cite[Definition A.20]{tignol2015value}. We will frequently use the fact that passing to the (strict) Henselization of a valued field does not change the value group~\cite[Corollary A.28]{tignol2015value}. 

An important example of a Henselian valuation is the $(t_0,\dots,t_{n-1})$-adic valuation
$\nu: F_n^*\to \bZ^{n}$ on the field of iterated Laurent series $F_n = k((t_0))\dots((t_{n-1}))$, where
$n$ is a positive integer; see \cite[Example A.16]{tignol2015value}. 
Here $\bZ^{n}$ is given the reverse lexicographic ordering with respect to the standard basis $e_0,\dots,e_{n-1}$ and
$\nu(t_i) = e_i$ for any $0\leq i\leq n-1$. The residue field of $(F_n,\nu)$ is naturally identified with $k$. In particular, $F_n$ is strictly Henselian if $k=k^{\sep}$.

We will denote the $d$-th Galois cohomology group of a $\Gal(F)$-module $M$ by $H^d(F,M)$.

%%%%%%%%%%%%%%%%%%%%%%%%%%%%%%%%%%%%%%%%%%%%%%%%%%%%%%%%%%%%%%%
%%%% To do - Move back the definition of H_p(F) and Q_p/Z_p from the appendix
%%%%       - Add notation for valuations
%%%%%%%%%%%%%%%%%%%%%%%%%%%%%%%%%%%%%%%%%

\section{The homomorphism \texorpdfstring{$\wedge \nu$}{nu} }\label{sect.construction_of_nu}% {the induced homomorphism}
Let $(F,\nu)$ be a valued field over $k$. That is, $F$ contains the base field $k$ and $\nu_{| k} = 0$ (remember that $\Char k \neq p$).
Denote the value group $\nu F$ by $\Gamma$.
Our first goal is to construct the homomorphism $H_p(F) \to \bigwedge  \Gamma / p^\infty $. 
Our starting point is the following homomorphism.

\begin{proposition}\label{prop.existenceofwedgenu}
There is a well-defined homomorphism of graded algebras:
$$\wedge' \nu : K(F) \to  \bigwedge  \Gamma,\ \ \{a_1,\dots,a_d\} \mapsto \nu(a_1)\wedge\dots\wedge\nu(a_d).$$
\end{proposition}

Our proof will make use of the fact that 
\begin{equation} \label{e.lang} \text{If $\nu(a) > \nu(b)$, then $\nu(a+b)=\nu(b)$;}
\end{equation}
see, e.g., \cite[p.~481]{lang2002algebra}.

\begin{proof}
Since the function $(F^*)^d \to \bigwedge^d  \Gamma, \ (a_1,\dots,a_d)\mapsto \nu(a_1)\wedge\dots\wedge\nu(a_d)$ is $\bZ$-multilinear, it factors through the tensor product to give homomorphisms:
$$(F^*)^{\otimes d} \to  \bigwedge^d  \Gamma,$$
for all $d\in \bN$. To check that these homomorphisms through $K(F)$ we need to verify 
the Steinberg relation, % holds. That is, we need to check that 
$$\nu(a)\wedge \nu(1-a) = 0, $$ 
for every $a\in F^*$. Consider three cases. 

\smallskip
Case 1. If $\nu(a) = 0$, then $\nu(a) \wedge \nu(1-a) = 0$ by bilinearity.

\smallskip
Case  2: If $\nu(a) > 0$, then $\nu(1 -a) = \nu(1) = 0$; see~\eqref{e.lang}. Hence, $\nu(a)\wedge \nu(1-a) = 0$.

\smallskip
Case 3. If $\nu(a) < 0$ then $\nu(1-a) = \nu(a)$ by \eqref{e.lang}. In this case, $\nu(a)\wedge \nu(1-a) = 0$ by anti-symmetry.
\end{proof}

% \begin{remark}
% We have used the fact that if $\nu(a) > \nu(b)$, then $\nu(a+b)=\nu(b)$. This follows from the following chain of inequalities:
% $$\nu(b) = \nu(a+b-a) \geqslant \min\{\nu(a+b),\nu(a)\} \geqslant \min\{\nu(b),\nu(a)\} = \nu(b),$$
% Which imply $\nu(b) = \min\{\nu(a+b),\nu(a)\}$.
% \end{remark}

\begin{remark}\label{rem.scalars}
It is clear from the definition of $\wedge' \nu$ that it is functorial. For any extension of valued fields $(F,\nu) \subset (\tilde{F},\tilde{\nu})$ and any element $\alpha \in K(F)$ we have
$$\wedge'(\tilde{\nu})(\alpha_{\tilde{F}}) = \wedge'\nu(\alpha),$$
under the natural identification of $\bigwedge \nu F$ with a subring of $\bigwedge \tilde{\nu} \tilde{F}$.
\end{remark}

Combining Proposition~\ref{prop.existenceofwedgenu} with Theorem~\ref{thm.norm}, we obtain: 

\begin{corollary}\label{cor.wedge.symbol}
There exists a well-defined homomorphism of graded abelian groups $\wedge\nu: H_p(F) \to  \bigwedge  \Gamma / p^\infty $ given on symbols by:
$$\wedge\nu(a_1,\dots,a_d)_{p^n} = \frac{1}{p^n} \otimes \nu(a_1)\wedge\dots\wedge \nu(a_d). $$
\end{corollary}

\section{First properties of $\wedge \nu$}

Let $(F,\nu)$ be a valued field over $k$ as in the previous section.
In this section we will explore some basic properties of $\wedge\nu$. To this end, it is convenient to choose a uniformizing parameter for $\nu$.
\begin{definition}
    A left inverse $\pi \colon \Gamma \to F^*$ to $\nu \colon F^* \to % \overset{\nu}{\to}  
    \Gamma$ will be called a \emph{uniformizing parameter}.
    Since the group operation in $\Gamma$ is written additively while $F^*$ is written multiplicatively, it will be convenient for us to use the exponential notation $\pi^{\gamma}$ in place of $\pi(\gamma)$, for any
    $\gamma\in  \Gamma$.
\end{definition}
Note that a uniformizing parameter always exists because $\Gamma$ is a free abelian group.
From now on we will fix a uniformizing parameter $\pi$ for $\nu$. Since elements in $H^1(F,\bZ/{p^n})$ anti-commute for any $n$, $\pi$ induces a left inverse $s_\pi : \bigwedge  \Gamma / p^\infty  \to H_p(F)$ given on generators by:
$$s_\pi(\frac{1}{p^n}\otimes \gamma_1 \wedge \dots \wedge \gamma_d) \mapsto (\pi^{\gamma_1},\dots,\pi^{\gamma_d})_{p^n}.$$
We call the image of $s_\pi$ the group of $\pi$-monomial classes. 
\begin{definition}\label{def.monomial}
    An element $\alpha \in H_p(F)$ will be called \emph{monomial} if it is of the form $\alpha = s_\pi(\omega)$ for some uniformizing parameter $\pi$ and class $\omega \in \bigwedge \Gamma/p^\infty$.
\end{definition}
It is important to keep in mind that any information captured by $\wedge\nu$ is contained in the subgroup of $\pi$-monomial classes, which is a small part of $H_p(F)$. Moreover, $\wedge\nu$ allows us to split $H_p(F)$ as a direct product of a divisible subgroup and the kernel of $\wedge\nu$. The following lemma gives a convenient generating set for this kernel. 

\begin{lemma}\label{lem.ker1}
    The kernel of $\wedge\nu : H_p^d(F) \to \bigwedge^d  \Gamma/ p^\infty $ is generated by symbols $(a_1,\dots,a_d)_{p^n}$ with $\nu(a_1) = 0$.
\end{lemma}

\begin{proof}
    Let $U \subset H_p^d(F)$ be the subgroup generated by symbols $(a_1,\dots,a_d)_{p^n}$ where $\nu(a_1)=0$. It is clear that $\wedge \nu(U) =0$. Note that any symbol $s= (a_1,\dots,a_d)_{p^n} \in H_p^d(F)$ is equivalent to a $\pi$-monomial class modulo $U$. Indeed, if we denote:
    $$ u_1 = (a_{1} \cdot \pi^{-\nu(a_{1})},a_{2},\dots,a_{d})_{p^n} \in U,$$
    then by multi-multiplicativity:
    $$ s-u_1 = (\pi^{\nu(a_1)},a_2,\dots,a_d)_{p^n} \in U$$
    Proceeding iteratively and using anti-symmetry we find elements $u_2,\dots,u_d \in U$ such that:
    $$ s- u_1 -\dots - u_d = (\pi^{\nu(a_1)},\dots,\pi^{\nu(a_d)})_{p^n}.$$
    Thus for any element $\alpha \in  H_p^d(F)$ mapped to $0$ under $\wedge\nu$ we can find an element $u\in U$ such that $\alpha - u$ is $\pi$-monomial. Since $\wedge\nu(\alpha)=\wedge\nu(u)=0$ and $s_\pi$ is a section of $\wedge\nu$ this implies:
    $$\alpha -u = s_\pi(\wedge\nu (\alpha -u)) =  0.$$
    Therefore $\alpha \in U$, as we wanted to show.
\end{proof}

Let $(\hat{F},\nu)$ be the strict Henselization of $(F,\nu)$. The passage to the Henselization may be viewed as a process of localization with respect to $\nu$. The next corollary shows $\wedge\nu$ gives a local description of $H_p(F)$. It can be easily deduced from a theorem of Wadsworth together with the norm-residue isomorphism theorem for strictly Henselian fields \cite[Proposition 2.1]{wadsworth1983henselian}. In \cite{brussel2001division}, Brussel used a similar description to study the Brauer group of a strictly Henselian field; see also \cite[Chapter 6]{tignol2015value}.

\begin{corollary}\label{cor.ker2}
    Denote by $H_p(\hat{F}/F) \subset H_p(F)$ the subgroup of elements split by $\hat{F}$. There is an exact sequence:
    $$ 0 \to H_p(\hat{F}/F) \to H_p(F) \overset{\wedge\nu}{\to}  \bigwedge \Gamma / p^\infty \to 0.$$
    In particular, if $(F,\nu)$ is strictly Henselian, then $\wedge\nu$ is an isomorphism.
\end{corollary}
\begin{proof}
    Since $(\hat{F},\nu)$ has the same value group as $(F,\nu)$, one has for all $\alpha \in H(F)$:
    $$\wedge\nu(\alpha) = \wedge\nu(\alpha_{\hat{F}}).$$
    Therefore $H_p(\hat{F}/F) \subset \ker \wedge\nu$. For the reverse inclusion, it suffices to check that symbols of the form $(a_1,\dots,a_d)_{p^n}$ with $\nu(a_1) =0$ are split by $\hat{F}$ by Lemma~\ref{lem.ker1}. This follows from the fact that if $a\in \hat{F}$ and $\nu(a) = 0$, then $a$ has an ${p^n}$-th root in $\hat{F}$. The residue class of $a$ has a ${p^n}$-th root in the residue field of $\hat{F}$ because it is separably closed. The residue field has the same characteristic as $k$ because $\nu(k^*)=0$ and so this ${p^n}$-th root may be lifted to $\hat{F}$ using Hensel's lemma.
\end{proof}

%%%%%%%%%%%%%%%%%%%%%%%%%%%%%%%%%%%%%%%%%%%%%%%%%%%%%%%%%%%%%%%%%%%%%%%%%%%%%
\section{A lower bound on essential dimension}
Now that we have constructed the invariant $\wedge\nu$, our goal is to show that if $\wedge\nu(\alpha)$ is "complicated enough" for $\alpha\in H_p(F)$, then $\alpha$ cannot descend to a subfield $F_0\subset F$ of small transcendence degree over $k$. The first step is to define a measure of complexity for elements of $\bigwedge \Gamma/ p^\infty $.

\begin{definition} 
    The width $\rho(\omega)$ of an element $\omega \in \bigwedge \Gamma/ p^\infty $ is defined as follows:
    $$\rho (\omega) = \min \bigg\{ \rank_\bZ W  \mid \  \begin{array}{cc}
           \text{Subgroups} \ W \subset \Gamma, \\
          \text{such that }  \omega \in \bigwedge W / p^\infty         
        \end{array}  \bigg\}.$$
\end{definition}
This definition is clearly analogous to the definition of essential dimension. Our next proposition makes the analogy precise. In the next section we will explain how to compute $\rho(\omega)$ in general. 

\begin{proposition}\label{prop.general_lower_bound}
    Let $(F,\nu)$ be a valued field with residue field $k$. Assume $\Char k \neq p$. Then
    
    \smallskip
    (a) $\ed_{k}(\beta;p)\geqslant \rho(\wedge\nu(\beta))$ for any $\beta \in H_p(F)$
    
     \smallskip
     (b) Moreover, if $\beta$ is monomial (see Definition~\ref{def.monomial}), then 
    $\ed_k(\beta) =\ed_k(\beta;p) =\rho(\wedge(\beta))$.
\end{proposition}

\begin{proof}
(a)    Let $F\subset L$ be a prime to $p$ extension such that $\ed_k(\beta_L) = \ed_{k}(\beta;p)$. By Lemma~\ref{lem.prime_to_p1} we can choose an extension $\tilde{\nu}$ of $\nu$ to $L$ such that $[\tilde{\nu} L : \nu E]$ is prime to $p$. By Lemma~\ref{lem.prime_to_p2} for any subgroup $W \subset \tilde{\nu}L $ we have: 
    \begin{equation}\label{eq.1}
        \bigwedge W\cap \nu E / p^\infty  = \bigwedge W / p^\infty .
    \end{equation}
    This implies: 
    \begin{equation}\label{eq.2}
        \rho(\wedge \nu (\beta)) = \rho(\wedge \tilde{\nu}(\beta_L)).
    \end{equation}
    Assume that $\beta_L$ descends to $k\subset L_0\subset L$. Then 
    $\wedge \tilde{\nu}(\beta_L) \in \bigwedge \tilde{\nu} L_0 / p^\infty$, and~\eqref{eq.2} gives
    \begin{equation} \label{eq3}
    \rank (\nu L_0 ) \geqslant \rho(\wedge \tilde{\nu}(\beta_L)) = \rho(\wedge \nu (\beta)).
    \end{equation}
    Now recall that by~\cite[Chapter XVII, Section 4, Theorem II]{hodge1954}, $\trdeg_k(L_0) \geqslant  \rank (\nu L_0 )$;
    see also \cite[Chapter VI, Theorem 3, Corollary 1]{zariski1960commutative}, \cite[Theorem 3.1]{meyer2012valuation}. Combining this with~\eqref{eq3}, we 
    obtain
    $\trdeg_k(L_0) \geqslant \rho(\wedge \nu (\beta))$.
    Therefore, $\ed_{k}(\beta;p)= \ed_k(\beta_L) \geqslant \rho(\wedge \nu (\beta))$. 

    \medskip
    (b) Now assume that $\beta$ is monomial. Then there exists a uniformizer $\pi$ of $\nu$ and a class $\omega \in \bigwedge \Gamma/ p^\infty $ such that $s_\pi(\omega) = \beta$. By definition we can find a subgroup $W\subset \Gamma$ such that $\omega \in  \bigwedge W/ p^\infty $ and
    $$\rank(W) = \rho(\wedge \nu (\beta)).$$
    Choose a basis $e_1,\dots,e_r$ of $W$ and set $F_0=k(\pi^{e_1},\dots,\pi^{e_r})$. Since $\omega \in  \bigwedge W/ p^\infty $ and $W = \tilde{\nu} F_0$, we have:
    $$\beta = s_\pi(\omega) \in \im (H_p(F_0) \to H_p(F)).$$
    We conclude that if $\beta$ is monomial, then
    $$\ed_k(\beta) \leqslant  \trdeg_k(F_0) \leqslant  r = \rho(\wedge \nu (\beta)).$$
Combining this inequality with part (a) and remembering~\eqref{e.ed-vs-ed-at-p}, we deduce part (b).
\end{proof}

%%%%%%%%%%%%%%%%%%%%%%%%%%%%%%%%%%%%%%%%%%%%%%%%%%%%%%%%%%%%%%%%%%%%%%%%
\section{Computing \texorpdfstring{$\rho$}{rho}}\label{sect.computing_rho}
In this section we show $\rho$ is relatively easy to compute using linear algebra. Our main tools are the contraction maps of $\bigwedge \Gamma$. Contraction is a way of ``applying" an element of $\bigwedge \Gamma$ to an element of
the dual group $\Gamma^* = \Hom_\bZ(\Gamma,\bZ)$ that generalizes multiplying a vector by a matrix (this special case will play a key role in Section~\ref{sec.brauer}). %11.9

\begin{definition}\label{def.contraction}
For any $\eta\in \bigwedge \Gamma^*$, we will write $\iota_{\eta} : \bigwedge \Gamma/p^\infty \to \bigwedge \Gamma/p^\infty$ for the tensor product $\id_{\bQ_p/\bZ_p}\otimes\iota'_\eta$ where $\iota'_\eta: \bigwedge \Gamma \to \bigwedge \Gamma$ is the usual contraction map from linear algebra, see \cite[Chapter III, p.~602]{bourbaki_ac_1-7}. We recall the explicit formula for $\iota_\eta$ because it will be used often. For any $f \in \Gamma^*$, $n\in \bN$ and $\gamma_1,\dots,\gamma_d \in \Gamma$ we have:
\begin{equation*}
    \iota_f(\frac{1}{p^n}\otimes \gamma_1\wedge\dots\wedge \gamma_d) = \sum_{i=1,\dots,d}  \frac{f(\gamma_i) (-1)^{i-1}}{p^n}\otimes  (\gamma_1\wedge\dots\hat{\wedge \gamma_i \wedge}\dots\wedge \gamma_d).
\end{equation*}
Here $\hat{\gamma_i}$ means "omit $\gamma_i$". This formula determines $\iota_f$ uniquely by linearity together with the condition $\iota_f(\frac{1}{p^n}\otimes 1_{\bigwedge \Gamma}) = 0$, for all $n \in \bN$. For general $\eta\in \bigwedge \Gamma^*$, $\iota_\eta$ is defined by linearity in $\eta$ and by the formula:
$$\iota_{f_1\wedge\dots\wedge f_d} = \iota_{f_1}\circ \dots \circ \iota_{f_d}.$$
We use the convention $\iota_{1}(\cdot) =\id_{\bigwedge\Gamma}$, where $1=1_{\bigwedge \Gamma^*}$ is the unit of $\bigwedge \Gamma^*$.
Note that for $\eta\in \bigwedge^d \Gamma^*$ and $\omega \in \bigwedge^t \Gamma/p^\infty$ we have $\iota_\eta(\omega) \in \bigwedge^{t-d} \Gamma /p^\infty$. In particular, if $d > t$, then $\iota_\eta(\omega) =0$. 
For any $\omega \in \bigwedge \Gamma/ p^\infty $, we denote by $A_\omega \subset \Gamma/ p^\infty $ the finite subgroup generated by degree $1$ parts $[\iota_{\eta}(\omega)]_1$ as we vary over $\eta\in \bigwedge \Gamma^*$.
\end{definition}

The next proposition shows that we can compute $\rho(\omega)$ using contractions.

\begin{proposition}\label{prop.computing_rho}
    Let $\omega \in \bigwedge \Gamma/ p^\infty $. The following hold:
    \begin{enumerate}
        \item If $W\subset \Gamma$ is a subgroup such that $\omega\in \bigwedge W/ p^\infty $, then $ A_\omega \subset W/ p^\infty $.
        \item The minimal number of generators of $A_\omega$ is $\rho(\omega)$.
        \item $\rho(\omega) = \dim_{\bF_p} A_\omega/p$.
    \end{enumerate} 
\end{proposition}
\begin{proof} (1) It is clear from the definition that for any $\eta\in \bigwedge \Gamma^*$ 
    and subgroup $W\subset \Gamma$:
    $$\iota_\eta \bigwedge W/ p^\infty  \subset  \bigwedge W/ p^\infty .$$
    Therefore if $\omega\in \bigwedge W/ p^\infty $, then $\iota_\eta(\omega)\in \bigwedge W/ p^\infty $. Letting
    $\eta$ vary over $\bigwedge \Gamma^*$, we see that $A_\omega \subset W/ p^\infty $.

    \smallskip
   (2)  Assume that $\omega \in \bigwedge W/ p^\infty $ for some subgroup $W\subset \Gamma$.
    By part (1), $A_\omega \subset W/ p^\infty $. For some $n$, we have $p^n\omega = 0$ and so:
    $$A_\omega \subset \frac{1}{p^n}\bZ_p/\bZ_p \otimes W \cong W/p^n W.$$
    Therefore $A_\omega$ is generated by $\rank W$ elements \cite[Proposition A.36]{tignol2015value}. Taking $\rank W$ to be minimal we get that $A_\omega$ is generated by $\rho(\omega)$ elements. 
    
    It remains to show that $A_\omega$
    cannot be generated by fewer than $\rho(\omega)$ elements. 
    Suppose that $A_\omega$ is generated by $m$ elements. Our goal is to prove $\rho(\omega)\leqslant  m$. Let $u_1,\dots,u_m \in \Gamma$ and $n\in \bN$ be elements such that:
    \begin{equation}\label{eq.rank0}
        A_\omega = \langle \frac{1}{p^n}\otimes u_i \mid i=1,\dots,m \rangle.
    \end{equation}
    Let $U$ be the saturation of the subgroup generated by $u_1,\dots,u_m$ in $\Gamma$. Then
    \begin{equation}\label{eq.rank1}
        \rank U \leqslant  m. 
    \end{equation}
    Now, we use the fact that any saturated subgroup of $\Gamma$ is a direct summand. Therefore we may assume that $\Gamma = \bZ^r$ with basis $e_1,\dots,e_r$ and $U = \langle e_1,\dots,e_s \rangle $, where $s = \rank U$. There are integers $a_{i_1,\dots,i_d}$ indexed by increasing sequences such that:
    $$\omega = \sum_{1\leqslant  i_1 <\dots<i_d \leqslant  r} \frac{a_{i_1,\dots,i_d}}{p^n}\otimes e_{i_1}\wedge\dots\wedge e_{i_d}.$$
    Let $j_1 < \dots < j_d$ be a sequence such  $j_d >   s$. It suffices to show $a_{j_1,\dots,j_d}$ is divisible by $p^n$ for any such sequence. % because  $U =\langle e_1,\dots,e_s \rangle $. 
    Let $e^1,\dots,e^r\in \Gamma^*$ be the dual basis of $\Gamma^*$. We denote:
    $$\eta = e^{j_1}\wedge \dots\wedge \hat{e^{j_t}}\wedge\dots \wedge e^{j_d},$$
    where $\hat{e^{j_t}}$ means ``omit $e^{j_t}$”, and compute the degree $1$ homogeneous part of $\iota_\eta(\omega)$:
    \begin{align*}
        [\iota_{\eta}(\omega)]_1 &= \sum_{1\leqslant  i_1 <\dots<i_d \leqslant  r}\frac{a_{i_1,\dots,i_d}}{p^n} \otimes [\iota_{\eta}(e_{i_1}\wedge\dots\wedge e_{i_d})]_1 \\
        &= \pm \frac{a_{j_1,\dots,j_d}}{p^n} \otimes  e_{j_d} + \sum_{i\neq j_d} \frac{b_i}{p^n} \otimes  e_i.
    \end{align*}
    for some $b_i \in \bZ$. Since $[\iota_{\eta}(\omega)]_1\in A_\omega$, \eqref{eq.rank0} implies there exists $u\in U = \langle e_1,\dots,e_s \rangle$ such that:
    $$[\iota_{\eta}(\omega)]_1 = \frac{1}{p^n} \otimes u.$$
    We conclude that $a_{j_1,\dots,j_d}$ is divisible by $p^n$ because $j_d  >   s$. This shows that $\omega \in \bigwedge U/ p^\infty $. Together with \eqref{eq.rank1} this gives $\rho(\omega) \leqslant   s= \rank U \leqslant  m$, as desired.
    % Therefore $\rho(\omega)$ is the minimal number of generators of $A_\omega$. 

    \smallskip
    (3) Since $A_{\omega}$ is an abelian $p$-group, the minimal number of generators of $A_{\omega}$ equals 
    $\dim_{\bF_p}( A_\omega /p)$.
\end{proof}

%%%%%%%%%%%%%%%%%%%%%%%%%%%%%%%%%%%%%%%%%%%%%%%%%%%%%%%%%%%%
\section{Proof of Theorem~\ref{thm.main} and Corollary~\ref{cor.infty}}
\label{sect.main}

Putting together the results of the previous two sections we are now able to prove our main result, Theorem~\ref{thm.main}.
We restate it here in slightly greater generality.

\begin{theorem}\label{thm.main_detailed}
    Let $(F,\nu)$ be a valued field over $k$ with value group $\Gamma$. Assume $\Char k\neq p$.
    Let $\alpha \in H_p(F)$ be a cohomology class, $\omega = \wedge\nu(\alpha)\in \bigwedge \Gamma/p^\infty$, and $A_\omega \subset \Gamma/p^\infty$ be the subgroup associated to $\omega$; see Definition~\ref{def.contraction}. Then
    \begin{equation}\label{eq.loose_ed}
        \ed_k(\alpha; p) \geqslant  \dim_{\bF_p}(A_\omega/ pA_\omega).
    \end{equation}
    Furthermore, if $\alpha$ is monomial (see Definition~\ref{def.monomial}), then
    \begin{equation}\label{eq.tight_ed}
    \ed_k(\alpha) = \ed_k(\alpha; p) = \dim_{\bF_p}(A_\omega/ pA_\omega).
    \end{equation}
    In particular, equality~\eqref{eq.tight_ed} holds whenever $(F,\nu)$ is strictly Henselian.
\end{theorem} 
\begin{proof}
    Both \eqref{eq.loose_ed} and \eqref{eq.tight_ed} follow immediately from Proposition~\ref{prop.general_lower_bound} and Proposition~\ref{prop.computing_rho}. If $(F,\nu)$ is strictly Henselian, then every class is monomial (see Corollary~\ref{cor.ker2}) and so \eqref{eq.tight_ed} holds in that case.
\end{proof}

We now turn to the proof of Corollary~\ref{cor.infty}. We need to show that for any $d \geqslant 2$ and any prime number 
$p$ there exist objects of $H^d_p(*)$ of arbitrarily large essential $p$-dimension. This follows from Proposition~\ref{prop.variables} 
below: just let $r \longrightarrow \infty$. 

\begin{proposition} \label{prop.variables} Let $r \geqslant 1$ and $d \geqslant 2$ be integers.
Let \[ F_{rd} = k((a_{11})) \ldots ((a_{1d}))((a_{21})) \ldots ((a_{2d})) \ldots ((a_{rd})) , \]
% \, | \, i = 1, \ldots, r, \; j = 1, \ldots, d)$, 
be the field of iterated Laurent series in $rd$ variables $a_{ij}$ over $k$ .  
% where $r \geqslant 1$ and $d \geqslant 2$.
Consider the class $\alpha \in H_p^d(K)$ given by
\[ \alpha = (a_{11}, \ldots, a_{1d})_{p^n} + \ldots + (a_{r1}, \ldots, a_{rd})_{p^n}. \]
Then $\ed_k(\alpha) = \ed_k(\alpha; p) = rd$.
\end{proposition}

\begin{proof}
Let $\nu: F_{rd}\to \bZ^{rd}$ be the $(a_{11},\dots,a_{rd})$-adic valuation on $F_{rd}$ (see Section~\ref{sect.notational-conventions}) and recall that
$\nu(a_{ij}) = e_{ij}$ for all $i,j$, where $e_{ij}$ is the natural basis of $\bZ^{rd}$.
We now evaluate $\omega = \wedge\nu(\alpha)$:
\[ \omega = \frac{1}{p^n} \otimes (e_{11} \wedge \ldots \wedge e_{1d} + e_{21} \wedge \ldots \wedge e_{2d} + \ldots + 
e_{r1} \wedge \ldots \wedge e_{rd}). \]
Denote by $e^{ij}\in ( \bZ^{rd})^*$ the dual basis. That is, for all $i,i',j,j'$:
$$e^{ij}(e_{i' j'}) = \begin{cases} 1 & \text{if } i = i' \ \text{and } j= j' \\ 
    0 & \text{otherwise}.
\end{cases}$$
For any $1\leqslant  i,j \leqslant  r$ let $\eta = e^{1j}\wedge\dots\wedge\hat{e^{ij}}\wedge \dots \wedge e^{dj} \in \bigwedge^{d-1}\Gamma^*$ be the wedge product where we omit $e^{ij}$. We have:
$$\iota_\eta(\omega) = \frac{1}{p^n}\otimes e_{ij}.$$
Therefore $ \bZ^{rd}/p^n \subset A_{\omega}$. The inclusion $A_{\omega} \subset \bZ^{rd}/p^n$ is checked easily using the definition of the contraction map. Therefore $A_\omega = \bZ^{rd}/p^n$ and $A_{\omega}/p \cong \bF_p^{rd}$. Since $\alpha$ is monomial, Theorem~\ref{thm.main_detailed} gives $\ed_k(\alpha) = \ed_k(\alpha;p)= rd$.
\end{proof} 

%%%%%%%%%%%%%%%%%%%%%%%%%%%%%
\section{Proof of Theorem~\ref{thm.brauer}}\label{sec.brauer} 
In this section we will deduce Theorem~\ref{thm.brauer} from Theorem~\ref{thm.main_detailed}. Let $(F,\nu)$ be a valued field over $k$  with value group $\Gamma= \bZ^r$.  There is a natural identification of $\bigwedge^2 \bZ^r / p^\infty $ with the skew-symmetric matrices in $M_r(\bQ_p/\bZ_p)$ by the homomorphism:
$$\frac{1}{p^n}\otimes u\wedge v\mapsto \frac{1}{p^n}( uv^t - vu^t). $$
Here we view $u, v \in \bZ^r$ as $r \times 1$ matrices (i.e. column vectors) 
and their transposes, $u^t$ and $v^t$ as $1 \times r$ matrices (i.e., row vectors) 
with integer entries. The products $u v^t$ and $v u^t$ are $r \times r$ matrices.

Under the above identification, the contraction map is given by matrix multiplication. That is, if $\omega \in \bigwedge^2 \bZ^r/ p^\infty $ is identified with the matrix $M \in M_r(\bQ_p/\bZ_p)$ and we identify $\Hom_\bZ(\bZ^r,\bZ)$ with $\bZ^r$ using the dot product, then:
$$\iota_{(\cdot)}(\omega): \bZ^r \to  \bigwedge^1 \Gamma/ p^\infty  = \bZ^r/ p^\infty  = (\bQ_p/\bZ_p)^r.$$
is given by $\iota_v(\omega) = Mv$. To see this, assume that $\omega = \frac{1}{p^n} u\wedge w$ for some $u,w\in \bZ^r$ and calculate using Definition~\ref{def.contraction}:
$$Mv = \frac{1}{p^n}(uw^t - wu^t)v = \frac{1}{p^n}(w^tv)u - \frac{1}{p^n}(u^tv)w = \iota_v(\omega). $$

We now proceed with the proof of Theorem~\ref{thm.brauer}. In fact, we will prove a slightly stronger
assertion, Theorem~\ref{thm.brauer'} below. If $k$ contains a primitive root of unity of degree $p^d$ for every
$d \geqslant 1$, then $\bQ_p/\bZ_p(2)$ is isomorphic to $\bQ_p/\bZ_p(1)$. For any field $F$ containing  $k$, 
$H_p^2(F) = H^2(F,\bQ_p/\bZ_p(2))$ is then isomorphic to the $p$-primary part of $\Br(F)$. Moreover, under this isomorphism
the symbols $(a,b)_{p^n}$ correspond to Brauer classes of cyclic algebras; see \cite[Proposition 4.7.1]{gille2017central}.

\begin{theorem} \label{thm.brauer'}
Let $(F,\nu)$ be a valued field over $k$ with value group $\bZ^r$ and $\Char k \neq p$. Assume $\alpha \in H_p^2(F)$ is a sum of symbols, $$\alpha = (a_1,b_1)_{p^n} + \dots + (a_k,b_k)_{p^n}.$$
% where $p$ is different from $\Char k$. 
Let $M \in M_r(\bZ)$ be the following skew-symmetric matrix:
$$M = \sum_{i=1,\dots,k} \nu(a_i)\nu(b_i)^t - \nu(b_i)\nu(a_i)^t.$$
Assume $d_1 \mid d_2 \mid \dots \mid d_r$ are the elementary divisors of $M$ and $i_0$ is the largest subscript
such that ${p^n}$ does not divide $d_{i_0}$. The following hold

\smallskip
(a) $\ed_k(\alpha; p) \geqslant  i_0$. In particular, if $p^r$ does not divide $\det(M)$, then $\ed_k(\alpha)\geqslant  r$. 

\smallskip
(b) If $(F,\nu)$ is strictly Henselian, then $\ed_k(\alpha) = \ed_k(\alpha ;p) = i_0$.
\end{theorem}

\begin{proof}
Denote $\omega = \wedge\nu(\alpha)$. By Corollary~\ref{cor.wedge.symbol}:
$$\omega = \frac{1}{p^n} \otimes(\nu(a_1)\wedge \nu(b_1)+\dots+\nu(a_k)\wedge \nu(b_k)).$$
Under the identification of $\bigwedge^2 \bZ^r/p^\infty $ with skew-symmetric matrices in $M_{r}(\bQ_p/\bZ_p)$ we have:
$$\omega = \frac{1}{p^n} \sum_{i=1,\dots,k}\nu(a_i)\nu(b_i)^t - \nu(b_i)\nu(a_i)^t = \frac{1}{p^n}M.$$
% Denote the matrix on the right-hand side by $\frac{1}{p^n}M$. 
As explained above, we can identify $\Hom_\bZ(\bZ^r,\bZ)$ with $\bZ^r$ so that $\iota_v(\omega)$ is given by matrix multiplication:
\begin{equation}\label{eq.matrix_mult}
    \iota_v(\omega) = \frac{1}{p^n}Mv \in (\bQ_p/\bZ_p)^r.
\end{equation}
Set
$A_\omega = \langle\iota_v(\omega) \mid v\in \bZ^r\rangle \subset (\bQ_p/\bZ_p)^r$,
as in Definition~\ref{def.contraction}.
Let $d_1\mid d_2 \mid \dots \mid d_r$ be the elementary divisors of $M \in M_{r}(\bZ)$ and let $\displaystyle k_i := \frac{p^n}{\gcd({p^n},d_i)}$ for each $i=1,\dots,r$. Using~\eqref{eq.matrix_mult} and elementary group theory one sees that $$ A_\omega \simeq \bZ/k_1 \oplus \dots \oplus \bZ/k_r.$$
Therefore $\dim_{\bF_p}(A_\omega/p) = i_0$, where $i_0$ is the number of $k_{i}$'s that are different from $1$. Equivalently, $i_0$ is the largest integer such that ${p^n}$ does not divide $d_{i_0}$.

Now parts $(a)$ and $(b)$ follow from Theorem~\ref{thm.main_detailed}.
\end{proof}

\begin{example} 
Let $F_{2r} = k((a_{11})) ((a_{12})), ((a_{21})), (a_{22})) \ldots ((a_{r1})) ((a_{r2}))$, be the field of iterated Laurent series in $2r$ variables $a_{ij}$. Assume that $k$ contains a primitive root of unity of degree $p^n$.
By Proposition~\ref{prop.variables}, the Brauer class $[D] \in \Br(F)$
of the division algebra $D = (a_{11}, a_{12})_{p^n}\otimes \dots \otimes (a_{r1}, a_{r2})_{p^n}$ 
has essential dimension $2r$. This can also be seen from Theorem~\ref{thm.brauer}. Indeed, 
carrying out the algorithm outlined there, we see that in this case
\[ M = \begin{pmatrix}  B & 0 & \dots & 0 \\
                    0 & B & \dots & 0 \\
                    \hdotsfor{4} \\
                    0 & 0 & \dots & B \end{pmatrix} . \]
Here each entry represents a $2 \times 2$ block: $0$ stands for the $2 \times 2$ zero matrix, and 
$\displaystyle B = \begin{pmatrix} 0 & 1 \\ -1 & 0 \end{pmatrix}$. The diagonal $2 \times 2$ block $B$ appears $r$ times.
Clearly, $\det(M) = \det(B)^r = 1$ and Theorem~\ref{thm.brauer} tells us that $\ed_k \, [D]  = 2r$. 
\end{example}

\begin{corollary}\label{cor.ed_br_is_even}
    Let $(F,\nu)$ be a strictly Henselian valued field over a field $k$ with $\Char k \neq p$. Then $\ed_k(\alpha) = \ed_k(\alpha, p)$ is even for any a Brauer class $\alpha \in \Br(F)$.
\end{corollary}

\begin{proof}
    This follows from Theorem~\ref{thm.brauer} because non-zero elementary divisors of an anti-symmetric matrix have even multiplicity; see \cite[Chapter IV, Theorem 3]{newman1972integral}.
\end{proof}

\section{Some limitations of Theorem~\ref{thm.brauer}}

% In this section we will discuss some limitations of Theorem~\ref{thm.brauer}.

The following proposition shows that Theorem~\ref{thm.brauer} 
cannot be used to prove a lower bound greater than $2n$ on the essential dimension of a Brauer class 
of index $p^n$. This is because $\wedge\nu$ is unaffected by passing to the strict Henselization; see Remark~\ref{rem.scalars}.

\begin{proposition} \label{prop.ed_limit}
    Let $(F,\nu)$ be a strictly Henselian valued field over $k$ with value group $\Gamma\cong \bZ^r$. If $\alpha\in H^2_p(F)$
    is split by a field extension $L/F$ of degree $p^n$, then
    $$\ed_k(\alpha) \leqslant  2n.$$
\end{proposition}

\begin{proof}
    By Chevalley's extension theorem \cite[Theorem 3.1.1]{engler2005valued}, we can extend $\nu$ to a valuation $\tilde{\nu}\colon L^* \to \tilde{\Gamma}$. By Ostrowski's theorem \cite[Theorem A.12]{tignol2015value}, 
        $$[\tilde{\Gamma}:\Gamma] \mid [L:F] = p^n. $$
    % Using the Smith normal form
    By the Elementary Divisors Theorem~\cite[Theorem III.7.8]{lang2002algebra}, we there exists a basis $e_1,\dots,e_r$ of $\tilde{\Gamma}$ and integers $d_1,\dots,d_r$ such that $d_1 e_1,\dots,d_r e_r$ is a basis for $\Gamma$ and $\prod_i d_i = [\tilde{\Gamma}:\Gamma] \mid p^n$.
    Thus each $d_i$ is $\pm$ a power of $p$, and we assume without loss of generality that $d_i = \pm 1$ for all $i > n$. For some $a_{ij}\in \bQ_p/\bZ_p$ we have:
    $$\wedge\nu(\alpha) = \sum_{1\leqslant  i<j \leqslant  r} a_{ij}\otimes (d_i e_i\wedge d_j e_j).$$
    Since $\alpha_L =0$, $\wedge\nu(\alpha)$ is in the kernel of the homomorphism:
    $$\phi: \bigwedge \Gamma/p^\infty \to \bigwedge \tilde{\Gamma} /p^\infty . $$
    Since $e_1,\dots,e_r$ is a basis of $\tilde{\Gamma}$, $\phi(\wedge\nu(\alpha))=0$ implies for all $i<j$:
    $$\phi( a_{ij}\otimes d _i e_i \wedge d_j e_j) = a_{ij}d_i d_j \otimes (e_i\wedge e_j) = 0 $$
    We conclude that $a_{ij} =0$ for all $n< i < j\leqslant  r$ because $d_i,d_j = \pm 1$.  Assume that $\displaystyle a_{ij} = \frac{b_{ij}}{p^N}$ for some integers $b_{ij}$. Then
    $$\wedge\nu(\alpha) = \sum_{1\leqslant  i\leqslant  n}\frac{1}{p^N} \otimes d_ie_i \wedge v_i,$$
    where $v_i = \sum_{i< j\leqslant  r} b_{ij}d_j e_j$. Let $W\subset \Gamma$ be the subgroup generated by $d_1 e_1,\dots,d_n e_n$ and $v_1,\dots,v_n$. Since $\wedge\nu(\alpha)\in \bigwedge W / p^\infty$, Theorem~\ref{thm.main_detailed} tells us that
    $\ed_k(\alpha) = \rho(\wedge\nu (\alpha)) \leqslant  2n$.
    \end{proof}

We will now show that the assumption that $\Char(k) \neq p$ in the statement of Theorem~\ref{thm.brauer'} cannot be dropped.

\begin{proposition} \label{prop.char-p} 
Let $k$ be a field of characteristic $p$. Then $\ed_k(\Br; p) = 2$.
\end{proposition}

\begin{proof} First we will show that $\ed_k(\Br; p) \leqslant 2$. In other words, $\ed_k(\alpha; p) \leqslant 2$ for  
every field $K$ containing $k$ and every $\alpha \in \Br(K)$. By the Primary Decomposition Theorem, 
we can write $\alpha = \alpha_p + \beta$, where the index of $\alpha_p$ is a power of $p$, 
and the index of $\beta$ is prime to $p$. Let $K'/K$ be a finite prime-to-$p$ extension which splits $\beta$.
Then $\alpha_{K'} = (\alpha_p)_{K'}$, and $\ed_k(\alpha; p) = \ed_k(\alpha_{K'}; p)$. 
Here the last equality is an easy consequence of the definition
of $\ed_k(\alpha; p)$; see, e.g., the proof of \cite[Proposition 1.5(2)]{merkurjev2009essential}.
After replacing $K$ by $K'$ and $\alpha$ by $\alpha_{K'}$, we may assume without loss of generality 
that the index of $\alpha$ is a power of $p$. By a Theorem of Albert, $\alpha$ is represented by a cyclic algebra over $K$;
see~\cite[Theorem 9.1.8]{gille2017central}. In other words, it lies in the image of the natural (functorial) morphism
\[ j_r \colon H^1(K, \bZ/p^r \bZ) \times H^0(K, \mathbb G_m) \to   \Br(K) \]
for some integer $r \geqslant 1$; see \cite[p. 260]{gille2017central}. Recall that  $H^1(K, \bZ/p^r \bZ)$ is in a functorial 
bijective correspondence with isomorphism classes of cyclic \'etale algebras $L/K$ of dimension $p^r$ and $H^0(K, \mathbb G_m) = K^*$.
If $\tau$ is a generator of the (cyclic) Galois group of $L/K$, then $j_r([L/K], [b])$ is the cyclic algebra generated by $L$ and $u$, 
subject to relations $u a  = \tau(a) u$ and $u^{p^r} = b$. Here $[L/K]$ is the class of $L/K$ in $H^1(K, \bZ/p^r \bZ)$ and
$[b]$ is the class of $b$ in $H^0(K, \mathbb G_m)$. Then
\begin{align*} \ed(\alpha; p) & \leqslant \ed_k( H^1(\ast,\bZ/ p^r \bZ) \times H^0(K, \mathbb G_m) ; \, p) \\
                              & \leqslant \ed_k( H^1(\ast,\bZ/ p^r \bZ); \, p) + \ed_k( H^0(\ast , \mathbb G_m); \, p ) \\
                               & = \ed_k(\bZ/ p^r \bZ; \, p)  + \ed_k( H^0(\ast , \mathbb G_m); \, p ). 
                               \end{align*}
Here the first inequality follows from the fact that $\alpha$ lies in the image of $j_r$ and the second 
from~\cite[Proposition 2.2]{merkurjev-survey}. The equality on the third line follows from the definition of $\ed_k(\bZ/ p^r\bZ; \, p)$.                             
Now observe that $\ed_k(\bZ/ p^r \bZ; \, p) = 1$ by \cite[Theorem 1]{reichstein-vistoli-prime}. On the other hand,
\[ \ed ( H^0(\ast , \mathbb G_m); \, p ) \leqslant  \ed ( H^0(\ast , \mathbb G_m) ) \leq 1. \]
Here the first inequality is a special case of~\eqref{e.ed-vs-ed-at-p}, 
and the second follows from the fact that an object $b \in H^0(K , \mathbb G_m)$ descends to $k(b)$. 
We conclude that $\ed(\alpha; p) \leqslant 1 + 1 = 2$, as desired.

It remains to show that $\ed_k(\Br; p) \geqslant 2$. Assume the contrary. Then
\[ \ed_{\overline{k}}(\Br; p) \leqslant \ed_k(\Br; p) \leqslant 1, \] 
where $\overline{k}$ is the algebraic closure of $k$; 
see, e.g.,~\cite[Proposition 1.5(2)]{merkurjev2009essential}.
After replacing $k$ by $\overline{k}$, we may assume that $k$ is algebraically closed. 
Choose a class $\alpha \in \Br(K)$ of index $p$ for some field $K$ containing $k$
(for example, the class of a universal division algebra of degree $p$).  
By our assumption $\ed_k(\alpha; p) \leqslant 1$, i.e., there exists a prime-to-$p$
field extension $L/K$ such that $\ed_k(\alpha_L) \leqslant 1$. In other words,
$\alpha$ descends to $\alpha_0 \in \Br(L_0)$ for some intermediate field
$k \subset L_0 \subset L$ such that $\trdeg_k(L_0) \leqslant 1$. 
By Tsen's Theorem, $\Br(L_0) = 0$; see~\cite[Lemma 6.2.3 and Proposition 6.2.8]{gille2017central}.
We conclude that $\alpha_{L_0} = 0$ and thus $\alpha_L = 0$ in $\Br(L)$, i.e., $L/K$
splits $\alpha$. This is a contradiction, because $\alpha$ has index $p$ over $K$, and
$[L:K]$ is prime to $p$. Thus $\ed_k(\Br; p) \geqslant 2$.
\end{proof}

%%%%%%%%%%%%%%%%%%%%%%%%%%%%%%%%%%%%%%%%%%%%%%%%%%%%%%%%%%%%
\section{A further consequence of Theorem~\ref{thm.main}}
\label{sect.exceptional_grps}

\begin{proposition}\label{prop.ed_non_symbol}
    Let $(F,\nu)$ be a strictly Henselian valued field over $k$. Let $\alpha \in H^d_p(F)$ for some $p\neq \Char k$. The following holds:
    \begin{enumerate}
        \item $\alpha \neq 0$ if and only if $\ed_k(\alpha;p) \geqslant  d$.
        \item $\alpha$ is not a symbol if and only $\ed_k(\alpha;p) \geqslant  d+2$.
    \end{enumerate}
    In particular, $\ed_k(\alpha)$ can never be equal to $d+1$.
\end{proposition}

\begin{proof}
Set $\omega := \wedge\nu(\alpha) \in \bigwedge^d \nu F/p^\infty$. By Theorem~\ref{thm.main_detailed},
\begin{equation}\label{eq.lemsymbol1}
\rho(\omega) = \ed_k(\alpha)=\ed_k(\alpha;p).
\end{equation}
Choose a subgroup $W\subset \nu F$ of rank $\rho(\omega)$ such that $\omega \in \bigwedge^d W/p^\infty$ .
% \begin{description}
%    \item[

  \smallskip
    $(1)$ If $\ed_k(\alpha;p) <  d$ then $\rank W= \rho(\omega) <  d$ and thus
    $\omega \in \bigwedge^d W/p^\infty =0$.
    Now Corollary~\ref{cor.ker2} tells us that $\alpha=0$.
    
    \smallskip
    $(2)$ If $\alpha = (a_1,\dots,a_d)_p$  is a symbol, then $\alpha$ descends to $k(a_1,\dots,a_d)$ and so $d\geqslant  \ed_k(\alpha;p)$. 
    
    Now assume that $d+1 \geqslant  \ed_k(\alpha;p)$. Our goal is to show that $\alpha$ is a symbol. By Corollary~\ref{cor.ker2}, it suffices to show $\omega$ is a pure tensor. By \eqref{eq.lemsymbol1}, we may enlarge $W$ to assume $\rank W = d+1$ and $\omega \in \bigwedge^d W/p^\infty$. Since $ \bigwedge^d W/p^\infty$ is generated by pure tensors it suffices to show the sum of two pure tensors is a pure tensor. A pure tensor in $\bigwedge^d W/p^\infty$ is of the form:
    $$\eta = \frac{a}{p^n} \otimes u_1\wedge \dots\wedge u_d,$$
    for some $u_1,\dots,u_d\in W$. Let $U = \langle u_1,\dots,u_d\rangle$ and denote by $U^{\sat}$ its saturation in $W$.  If $u'_1,\dots,u'_d$ are a basis of $U^{\sat}$, then $\eta = \frac{a'}{p^n}\otimes u'_1 \wedge \dots \wedge u'_d$ for some $a'\in \bZ$ (one can check this using the Smith normal form). Therefore we may assume $U$ is saturated. Let $\theta$ be another pure tensor:
    $$\theta = \frac{b}{p^n} \otimes v_1\wedge \dots\wedge v_d,$$
    for some $v_1,\dots,v_d\in W$ such that $V = \langle v_1,\dots,v_d\rangle$ is saturated in $W$. We will show $\eta + \theta$ is a pure tensor. Since $U$ and $V$ are saturated we have:
    $$(U\cap V)_{\bQ} = U_\bQ \cap V_\bQ,$$
    where $U_{\bQ} = U \otimes \bQ$ and similarly for $V_{\bQ}$ and $(U \cap V)_{\bQ}$.
    Therefore,
    $$\rank_\bZ(U\cap V) = \dim_\bQ(U_\bQ \cap V_\bQ) \geqslant \dim_{\bQ} (U_{\bQ}) + \dim_{\bQ} (V_{\bQ}) -
    \dim(W_{\bQ}) = 2d - \dim_\bQ (W_\bQ) = d-1.$$
    Since $U\cap V$ is clearly saturated in both $U$ and $V$ we can complete a basis of $U\cap V$ to a basis of $U$ and to a basis of $V$. In other words, we may assume $u_i = v_i$ for all $1\leqslant  i\leqslant  d-1$. Then 
    \begin{align*}
        \eta + \theta &= \frac{a}{p^n} \otimes u_1\wedge \dots\wedge u_d + \frac{b}{p^n} \otimes u_1\wedge \dots\wedge u_{d-1}\wedge v_d\\
                      &= \frac{1}{p^n}\otimes u_1\wedge \dots\wedge (a u_d+b v_d).       
    \end{align*}
    The right hand side is a pure tensor. This shows that any element in $\bigwedge^d W/p^\infty$ is a pure tensor and thus completes the proof of part (b).
% \end{description}
\end{proof}

In light of Corollary~\ref{cor.ed_br_is_even}, one might guess that $\ed_k(\alpha)$ is divisible by $d$ for any $\alpha\in H^d_p(F)$. The following proposition shows that this is false for any $d \geqslant 3$.

\begin{proposition} \label{prop.possible_values_of_ed}
Let $F_n = k((t_0))\dots((t_{n-1}))$. Assume that $d \geqslant 3$. Then for any $n \geqslant  d + 2$ and $p\neq \Char k$,
there exists $\alpha \in H^d_p(F_n)$ such that $\ed_k(\beta;p) = n$. 
\end{proposition}

\begin{proof} We claim that 
\[ \alpha = \sum_{i_1 + \ldots + i_d \equiv 0 \pmod{n}} (t_{i_1})_p  \cup \ldots \cup (t_{i_d})_p  \] 
has the desired property. Here the sum is taken over ordered $d$-tuples $(i_1, \ldots, i_d)$ such that
$0 \leqslant i_1 < \ldots < i_d \leqslant n-1$ and $i_1 + \ldots + i_d \equiv 0 \pmod{n}$. Let $\nu$ be the $(t_0,\dots,t_{n-1})$-adic valuation on $F_n$ and set
  \[ \omega:=  \wedge\nu (\alpha) = \frac{1}{p}\otimes  \sum_{i_1 + \ldots + i_d \equiv 0 \pmod{n}}
  e_{i_1}\wedge \ldots \wedge e_{i_d}, \] 
where $e_0, e_1, \ldots, e_{n-1}$ is the standard basis in $\bZ^n$.
In view of Theorem~\ref{thm.main_detailed} it suffices to show that
$A_\omega = \bZ^{n}/p \subset \bZ^{n}/p^\infty$. 
Let $e^1,\dots,e^n\in (\bZ^n)^*$ be the dual basis. From the definition of $\omega$, we see that when we contract $\omega$ with a tensor of 
the form $e^{j_1} \wedge \ldots \wedge e^{j_{d-1}}$, where $0 \leqslant j_1 < \ldots < j_{d-1}$,
we obtain $\pm e_j$, where $0 \leqslant j \leqslant n-1$ is uniquely determined by the congruence
$j \equiv -j_1 - \ldots - j_{d-1} \pmod{n}$. By definition, every 
$e_j$ of this form lies in $A_{\omega}$. To complete the proof, we need to show that every basis vector $e_0, \ldots, e_{n-1}$
can be obtained in this way. Equivalently, it remains to prove the following.

\smallskip
{\bf Claim:} Assume $d \geqslant 3$ and $n \geqslant d + 2$ are integers. Then
for every $j \in \bZ/ n $, there exists a subset $S \subset \bZ/ n $ consisting of $d$ distinct 
elements, $i_1, \ldots, i_d$, such that $j \in S$ and $i_1 + \ldots + i_d = 0$ in $\bZ/n$.

\smallskip
To prove the Claim, let $\cP$ be the set of pairs $\{ i , - i \}$, where $2i \neq 0$ in $\mathbb Z/ n$. If $n$ is odd,
then $2i = 0$ is only possible in $\bZ/ n$ if $i = 0$, and thus $|\cP| = (n-1)/2$. If $n = 2m$ is even, then 
the equation $2i = 0$ has two solutions in $\bZ/n$, $i =0$ or $i = m$. In this case
$|\cP| = (n-2)/2 = m - 1$.

\smallskip
{\bf Case 1.} $2j \neq 0$ in $\bZ/ n$, and $d$ is even. Take $S$ to be the union of $d/2$ pairs from $\mathcal{P}$, one of these pairs being $\{ \pm j \}$.
This is possible because we are assuming that $n \geqslant d + 2$, and thus
$| \mathcal{P} | \geqslant (n-2)/2 \geqslant d/2$. 
% with $(d-2)/2 = d/2 -1$ other pairs from $\mathcal{P}$. 
% Note that $\mathcal{P}$ has $d/2-1$ pairs beside $\{ \pm j \}$ because $n \geqslant d + 2$ and thus 
% $|\mathcal{P}| \geqslant (n-2)/2 \geqslant d/2$. Each pair adds up to $0$ in $\bZ/ n$, so the sum of the elements of 
% of $S$ will be $0$, as desired. 

\smallskip
{\bf Case 2.} $2j \neq 0$ in $\bZ/ n$, and $d$ is odd. Take $S$ to be the union of $0$ and $(d - 1)/2$ 
pairs from $\mathcal{P}$, one of these pairs being $\{ \pm j \}$. Again, this is possible because 
$|\mathcal{P}| \geqslant d/2$.

% Here the construction is similar to Case 1: we combine $0, -j and j$ with  
% $(d-3)/2$ pairs from $S$ (other than $\{ \pm j \}$) to create $S$. Once again, $\mathcal{P}$ has $(d-3)/2$ pairs, 
% other than $\{ \pm j \}$, because $|\mathcal{P}| \geqslant (n-2)/2 \geqslant d/2 > (d-3)/2 + 1$.

\smallskip
{\bf Case 3.} $2j = 0$ and $d$ is odd. Take $S$ to be the union of $\{ 1, j, j-1 \}$ and $(d-3)/2$ pairs from $\mathcal{P}$,
other than $\{ \pm 1 \}$ and $\{\pm (j-1) \}$. Note that $1$, $j$ and $j-1$ are distinct in $\bZ/ n$; this readily follows from our assumption that $n \geqslant d + 2 \geqslant 5$. Note also that $(d-3)/2$ pairs, other than
$\{ \pm 1 \}$ and $\{\pm (j-1) \}$, exist in $\mathcal{P}$, because when $d$ is odd, 
the inequality $|\cP| \geqslant d/2$ can be strengthened to $|\cP| \geqslant (d+1)/2 = (d-3)/2 + 2$.

\smallskip
{\bf Case 4a.} $2j = 0$, $d$ is even, and $n$ is odd. In this case $j = 0$ in $\bZ/n$.
Moreover, for even $d$, the inequality $d \geqslant 3$ can be strengthened to $d \geqslant 4$, and for odd $n$,
the inequality $n \geqslant d +2 \geqslant 6$ can be strengthened to $n \geqslant 7$. Consequently, 
the four elements $0$, $1$, $2$ and $-3$ are
distinct in $\bZ/n$, and we define $S$ to be the union of $\{ 0, 1, 2, -3 \}$ and $(d-4)/2$ pairs from $\mathcal{P}$, other than
$\{ \pm 1 \}$, $\{ \pm  2 \}$ and $\{ \pm 3 \}$. Note that $(d-4)/2$ pairs with this property exist because for $n$ odd
$|\cP | = (n-1)/2 \geqslant (d+1)/2$. Since $d$ is even, this translates to 
$|\cP| \geqslant  (d+2)/2 = (d-4)/2 + 3$.

\smallskip
{\bf Case 4b.} $2j = 0$, $d$ is even, and $n = 2m$ is even. Here $j = 0$ or $m$, and we define $S$ as the union of $\{ 0, 1,  m-1, m  \}$
and $(d-4)/2$ pairs from $\mathcal{P}$, other than $\{ \pm 1 \}$ and $\{ m-1, m+1 \}$. The elements $0$, $1$, $m$, $m+1$ 
are distinct in $\bZ/ n$, because $n \geqslant d + 2 \geqslant 6$, and $\cP$ has $(d-4)/2$ pairs, other than 
$\{ \pm 1 \}$ and $\{ m-1, m+1 \}$, because $|\cP| \geqslant d/2 = (d-4)/2 + 2$.
\end{proof}

\begin{remark}\label{ex.possible_values_of_ed}
    Similarly, in light of Proposition~\ref{prop.ed_limit}, one might guess that if $\alpha\in H^d_p(F)$ is split by a field extension $E/F$ of degree $n$, then $\ed_k(\alpha) \leqslant dn$. This is also false. For example, let $d = 3$, $F = k((t_0))\dots((t_{2r}))$, and 
    \[ \alpha = (t_0)_p \cup \big( (t_1, t_2)_p + \ldots + (t_{2r-1}, t_{2r})_p \big) . \]  
    Then $\alpha$ is split by $E = F(\sqrt[p]{t_0})$, where $[E:F] = p$. On the other hand,
    contracting 
    \[ \omega:=  \wedge\nu (\alpha) = \frac{1}{p}\otimes e_0 \wedge [ e_{1}\wedge  e_{2}  + \ldots + e_{2r-1} \wedge e_{2r}] \]
with tensors of the form $e^{j_1} \wedge e^{j_2}$, as we did in the proof of Proposition~\ref{prop.possible_values_of_ed},
we see that $A_\omega = \bZ^{2r + 1}/p \subset \bZ^{2r + 1}/p^\infty$.
By Theorem~\ref{thm.main_detailed}, $\ed_k(\alpha) = \ed_k(\alpha; p) = 2r + 1$. Since $r$ can be taken to be an 
arbitrary positive integer, the ratio
$\displaystyle \frac{\ed_k(\alpha)}{[E:F]} = \frac{2r + 1}{p}$ can be arbitrarily large.
\end{remark}

\section{Anisotropic torsors}

In this section take a slight detour to prove a structural result about anisotropic $E_8$-torsors, Proposition~\ref{prop.ed_3_E8} below. This proposition will be used
in the proof of Theorem~\ref{thm.rost} in the next section. % we also feel that it is of independent interest.

 Let $G$ be a connected quasi-split reductive group and $T\in H^1(F,G)$ a $G$-torsor over a field $F\in \Fields_k$. For any subgroup $H\subset G$, we will say that $T$ \emph{admits reduction of structure to} $H$, if it is in the image of $H^1(F,H)\to H^1(F,G)$. The torsor $T$ is called \emph{isotropic} if it admits reduction of structure to a proper parabolic subgroup of $G$ (or equivalently, if the twisted group ${}_T G$ is isotropic as an algebraic group \cite[Lemma 2.2]{ofek2022reduction}). Otherwise, $T$ is said to be \emph{anisotropic}. Recall that if $H$ is a finite constant group, $H^1(F,H)$ classifies Galois etale algebras of $F$ whose Galois group is a homomorphic image of $H$. The next lemma is inspired by~\cite[Proposition 7.2]{gille2009lower}. 

\begin{lemma}\label{lem.aniso} 
    Let $(F,\nu)$ be a strictly Henselian field with residue field $k$ and valuation ring $\cO$. Let $G$ be a connected quasi-split reductive group over a field $F_0 \subset \cO$. Assume that $G_\cO$ contains a constant finite subgroup $\iota: H\hookrightarrow G_{\cO}$ such that $H_{\overline{k}}$ is not contained in any proper parabolic subgroup of $ G_{\overline{k}}$. Then for any torsor $T \in H^1(F,H)$  corresponding to an $H$-Galois field extension $L/F$, the push-forward $\iota_*(T) \in H^1(F,G)$ is anisotropic. 
\end{lemma}
\begin{proof}
    By assumption, there exists an isomorphism $c: \Gal(L/F)\to H$ and $\iota_*(T)$ is represented by the cocycle $\iota \circ c$.
    Assume that $\iota_*(T)$ admits reduction of structure to a parabolic subgroup $P\subset G$. Our goal is 
    to show $P=G$. Since $\iota_*(T)$ admits reduction of structure to $P$, there exists $g\in G(L)$ such that for all $\sigma\in \Gal(L/F)$:
    \begin{equation}
        g^{-1} c_\sigma {}^\sigma g\in P(L).
    \end{equation}
    Setting $x \in (G/P)(L)$ to be the coset represented by $g$, we see that for all $\sigma\in \Gal(L/F)$:
    \begin{equation}\label{e.aniso_E8_1}
        c(\sigma) {}^\sigma x = x.
    \end{equation}
    Since $\nu$ is Henselian, it has a unique extension to $L$. By abuse of notation, we will continue to denote this extended valuation by $\nu$. We will write $\cO_L$ for the valuation ring of $\nu$ in $L$ and $l$ be the residue field. Since $k$ is separably closed, $l$ is a purely inseparable extension of $k$. Since $G/P$ is projective, $x \in (G/P)(L)$
    can be lifted to a point in $(G/P)(\cO_L)$. By abuse of notation, we will simply say that $x \in (G/P)(\cO_L)$. Note that $\Gal(L/F)$ preserves $\nu$ because $L$ is Henselian. Therefore $\Gal(L/F)$ acts on $\cO_L$. The reduction homomorphism $\cO_L \to l$ is $\Gal(L/F)$-equivariant with respect to the trivial $\Gal(L/F)$-action on $l$ because $l/k$ is purely inseparable. Let $\pi: G(\cO_L)\to G(l)$ be the induced homomorphism. Applying $\pi$ to \eqref{e.aniso_E8_1} gives:
    $$\pi(x) = \pi(c(\sigma){}^\sigma x ) = \pi(c(\sigma))\pi(x).$$
    Since $H \subset G_{\cO}$ is constant, $H(\cO_L) = H(l)$ and the restriction of $\pi$ to $H$ is the identity. Therefore we get:
    $$\pi(x) = \pi(c(\sigma){}^\sigma x ) = c(\sigma)\pi(x).$$
    Since $l$ is separably closed, we can find $g\in G(l)$ such that $\pi(x)= g P(l)$. Since $c$ is an isomorphism we get $H g P(l) =gP(l)$
    and therefore: $$H_l \subset g P_l g^{-1}.$$ 
    Since we are assuming that $H_{\overline{k}}$ is not contained in any proper parabolic of $G_{\overline{k}}$, we conclude that $g P_l g^{-1} = G_l$ and thus $P = G$. This finishes the proof. 
    %Should we add a comment about the fact that we use the dimension of the fibers G_x and P_x are constant 
    % because G and P are flat
\end{proof}
The next lemma gives a convenient sufficient condition for a subgroup $H\subset G$ to not be contained in any proper parabolic subgroup $P\subset G$.
\begin{lemma}\label{lem.finite_centralizer}
    Let $G$ be a connected semisimple group over a field $k$ and let $H\subset G$ be a finite subgroup. If the centralizer $C_G(H)$ is finite (i.e., 0-dimensional) and $|H|$ is coprime to $\Char k$, then $H$ is not contained in any proper parabolic subgroup of $G$.
\end{lemma}
\begin{proof}
    Assume that $H$ is contained in a parabolic subgroup $P\subset G$.
    Since $H$ is of order prime to $p$ it is linearly reductive over $k$. Therefore $H$ is contained in a Levi subgroup $L\subset P$ (see \cite[Lemma 11.24]{jantzen2004nilpotent}). If $P\neq G$, then the center $Z_L$ of $L$ contains a split torus \cite[Proposition 20.6]{borel_lag}. This contradicts our assumption that $C_G(H)$ is finite because 
    clearly $Z_L \subset C_G(H)$. Therefore we must have $P=G$.
\end{proof}

We are now ready to proceed with the main result of this section. 
%  This proposition will be used in the proof of Theorem~\ref{thm.rost} in the next section; we also feel that it is of independent interest.

\begin{proposition}\label{prop.ed_3_E8}
Let $G$ be a split group of type $E_8$ over a field $k$ and assume $\Char k \neq 2,3$. Let $T \in H^1(F,G)$ be a torsor over a strictly Henselian field $F\in \Fields_k$ over $k$. The following are equivalent:
    \begin{enumerate}
        \item $\ed_k(T;3)\geqslant  5$.
        \item $\ed_k(R_G(T);3)\geqslant  5$.
        \item $T_L$ is anisotropic for any prime-to-$3$ extension $F\subset L$.
    \end{enumerate}
\end{proposition}

\begin{proof}
Passing from $F$ to its prime-to-$3$ extension $F^{(3)}$ does not affect the validity of 
assertions (1), (2) and (3). Thus we may assume without loss of generality that $F$ is a $3$-special field.
In particular, since $\Char F\neq 3$, this implies that $F$ is perfect. By assumption, $F$ is equipped with a Henselian valuation $\nu$ that is trivial on $k$. Therefore the residue field of $\nu$ has the same characteristic as $k$. Hensel's lemma implies that $F$ contains all primitive roots of unity of order $3^n$ for $n\in\bN$. Let $k\subset k'$ be the field generated by all roots of unity in $F$. We replace $k$ by $k'$ without effecting assertions $(1)$ and $(2)$ by \eqref{e.ext_of_scalar}. This allows us to view the Rost invariant as a natural homomorphism $H^1(\ast,G)\to H^3_3(\ast)$ because $\bQ_3/\bZ_3(2)\cong \bQ_3/\bZ_3(3)$ as Galois modules over $k$.

\smallskip
    (3) $\iff$ (2)  By Proposition~\ref{prop.ed_non_symbol}, $\ed_k(R_G(T);3) \geqslant  5$ if and only if $R_G(T)$ is not a symbol. Since $F$ is $3$-special, \cite[Theorem 10.24]{garibaldi2016shells} implies that $R_G(T)$ is a symbol if and only if $T$ is isotropic. Therefore $T$ is anisotropic if and only if $\ed_k(R_G(T);3)\geqslant  5$

\smallskip
    (2)$\implies$ (1) follows directly from \eqref{eq.lowerbound_coh_invariant}.

\smallskip
    (1)$\implies$ (3) We assume $T$ is isotropic and show $\ed_k(T;3)\leqslant  4$. Since $T$ is isotropic, it admits reduction of structure to some proper parabolic subgroup $P\subset G$. Assume without loss of generality that $P$ is maximal amongst the proper parabolic subgroups.  Let $L\subset P$ be a Levi subgroup of $P$ and let $L'\subset L$ be the corresponding derived subgroup (it exists by \cite[Theorem 25.6]{milne2017algebraic}). Since $F$ is perfect, $H^1(F,L) \to H^1(F,P)$ is bijective \cite[Lemme 1.13]{sansuc1981groupe}. Note that $Z(L)^{\circ}$ is a split torus because it is contained in a maximal torus of $G$.  Therefore,
    $L/L'$ is a split torus as a homomorphic image of $Z(L)^{\circ}$ \cite[Chapter V, Theorem 15.4]{borel_lag}, and $H^1(F,L')\to H^1(F,L)$ is surjective by Hilbert $90$. Therefore the composition $H^1(F,L') \to H^1(F,P)$ is surjective and $T$ admits reduction of structure to $L'$. Therefore $\ed_k(T;3)\leqslant  \ed_k(L';3)$ and it suffices to prove:
    \begin{equation}\label{eq.ed_of_levi}
        \ed_k(L';3) \leqslant  4.
    \end{equation}
    Since $P$ is a maximal proper parabolic of $G$, the Dynkin diagram of $L'$ is obtained from the Dynkin diagram of $E_8$ by removing one vertex \cite{tits1965classification}. By \cite[Corollary 5.4(b)]{springer-steinberg},
    $L'$ is simply connected. Hence, $L'$ is a product of simple simply connected groups.
    Removing each of the eight vertices in the Dynkin diagram of $E_8$, we obtain the following possibilities for $L'$.

   \smallskip
   (i) $D_7$, (ii) $A_1\times A_6$, (iii) $A_2\times A_1\times A_4$, (iv) $A_7$, (v) $A_4\times A_3$, 
   (vi) $D_5\times A_2$, 

   \smallskip
   (vii) $E_6\times A_1$, (viii) $E_7$.

   \smallskip \noindent
   That is, in case (i), $L'$ is isomorphic to $\Spin_{14}$, it case (ii) to $\SL_2 \times \SL_7$, etc.
    We will now check that the the inequality~\eqref{eq.ed_of_levi} holds in each of these eight cases.
    Recall that if $G = H_1\times \dots\times H_r$, then 
    $$\ed_k(G;3) \leqslant  \ed_k(H_1;3) + \dots + \ed_k(H_r;3); $$
    see, e.g.,~\cite[Proposition 3.3]{merkurjev-survey}. The simply connected simple group of type $A_n$ is
    $\SL_n$. This group has essential dimension $0$. In cases (ii) - (v), $L'$ is a direct product of simply connected simple groups of type $A$ and thus $\ed_k(L') = \ed_k(L'; 3) = 0$.
    The simply connected group of type simple group of type $D_n$ is $\Spin_{2n}$. Since $3$ is not a torsion prime
    for these groups, we have $\ed_k(\Spin_{2n}; 3) = 0$; see~\cite[Proposition 3.16]{merkurjev-survey}. We conclude that $\ed_k(L'; 3) = 0$ in cases (i) and (vi) as well. In case (vii), we have 
    \[ \ed_k(L'; 3) = \ed_k(E_6^{\rm sc} \times \SL_2; 3) \leqslant \ed_k(E_6^{\rm sc}; 3) + \ed_k(\SL_2; 3) =
    \ed_k(E_6^{\rm sc}; 3) = 3 . \]
For the last equality, see~\cite[Section 3h]{merkurjev-survey}. In case (viii),
\[ \ed_k(L'; 3) = \ed_k(E_7^{\rm sc}; 3) = 4, \]
where the last equality is again taken from~\cite[Section 3h]{merkurjev-survey}.  We conclude that~\eqref{eq.ed_of_levi} holds in each of the cases (i) - (viii).   
\end{proof}

\section{Proof of Theorem~\ref{thm.rost} and a remark on spin groups}
\label{sect.rost}

For the reader's convenience, we begin by reproducing the statement of Theorem~\ref{thm.rost}.
% \begin{theorem*} 

\smallskip
{\bf Theorem~\ref{thm.rost}.} Let $k$ be a base field of characteristic $\neq 2$ and 
$F_n$ be the iterated Laurent series field $F_n = k((t_0))((t_2)) \ldots ((t_{n-1}))$. Then there exist
(i) an $E_7^{\rm sc}$-torsor $T_1 \to \Spec(F_7)$, (ii) an $E_8$-torsor $T_2 \to \Spec(F_9)$, and (iii) an $E_8$-torsor $T_3 \to \Spec(F_5)$
such that (i) $\ed_k(R_{E_7^{\rm sc}}(T_1); 2) = 7$, (ii) $\ed_k(R_{E_8}(T_2); 2) = 9$, and (iii) $\ed_k(R_{E_8}(T_3); 3) = 5$, respectively.

Here in (iii) we are assuming that $k$ contains a primitive $3$rd root of unity. 
% \end{theorem*}

\begin{proof}
We assume at first that $k$ is separably closed. In this case, The functors $H^3(\ast,\bQ_p/\bZ_p(2))$ and $H^3_p(\ast)$ are naturally isomorphic because $\bQ_p/\bZ_p(2)$ and $\bQ_p/\bZ_p(3)$ are isomorphic as Galois modules for all $p\neq \Char k$. Therefore we can and shall see $R_G$ as taking values in $H^3_p$ where $p =2$ in parts (i),(ii) and $p=3$ in part (iii).

 \smallskip
\textbf{Part (i):} Let $G$ be a simply connected split group of type $E_7$ over $k$. Let $L = k(s_0,\dots,s_6)$ be a field of rational functions in $7$ independent variables and $H$ a quasi-split simply connected group of type $E_6$ over $L$ that is split by the extension $L(\sqrt{s_0})$.
In \cite[Section 6.2.3]{chernousov2003kernel}, Chernousov constructed a torsor $S\in H^1(L,H)$ such that:
$$R_H(S) = (s_0)_2\cup [(s_1,s_3)_2+ (s_2 s_3 s_5,s_4)_2 + (s_5,s_6)_2].$$
Moreover, there is an embedding $\iota :H\to G$ such that $R_G(\iota_*(S))= R_H(S)$ \cite[Proposition 3.6]{garibaldi2001rost} (see also \cite[Example A.3]{garibaldi2003cohomological}). 
We make a rational change of variables $t_0 = s_0$, $t_1=s_1, t_2= s_3, t_3= s_2 s_3 s_4, t_4= s_4, t_5= s_5, t_6 = s_6$. 
Clearly $L=k(t_0,\dots,t_6)$ and we have:
$$R_G(\iota_*(S)) = (t_0)_2 \cup [(t_1,t_2)_2 + (t_3,t_4)_2 + (t_5,t_6)_2].$$
We set $T_1 = \iota_*(T_0)_{F_7}$ and observe that $\ed_k(R_G(T_1);2) = 7$ by Remark~\ref{ex.possible_values_of_ed}.

\smallskip
\textbf{Part (ii):} Let $G$ be a split group of type $E_8$ over $k$. In \cite[Corollary 11.3]{garibaldi2009orthogonal} Garibaldi proves there exists a torsor $T_2\in H^1(F_9,G)$ such that:
$$R_G(T_2) = (t_0)_2\cup [(t_1 ,t_2)_2 + \dots + (t_7 ,t_8)_2] .$$
We have $\ed_k(R_G(T_2);2)  = 9$ by Remark~\ref{ex.possible_values_of_ed} again.

\smallskip
\textbf{Part (iii):} Let $G= E_8$ as in the previous part.  Denote $H := \bZ^5/3$. By \cite[Lemma 11.5]{griess1991elementary}, there is an embedding $\iota :H\hookrightarrow G$ such that $\iota(H)$ is self-centralizing. Let $K_5 = k(t_0,\dots,t_4)$ be a rational field in $5$ variables. We denote by $S \in H^1(K_5,H)$ the torsor corresponding to the $H$-Galois extension $L_5 = K_5(t_0^{1/3},\dots,t_4^{1/3})$ of $K_5$. Let $T_3 = \iota_*(S)_{F_5}$. Inequality \eqref{e.morphism} implies:
$$\ed_k(R_{G}(T_3);3) \leqslant \ed_k(R_{G}(T_3)) \leqslant \ed_k(T_3) \leqslant \trdeg_k(K_5) = 5.$$
Let $\nu$ be the $(t_0,\dots,t_4)$-adic valuation on $F_5$. Since $k$ is separably closed, $(F_5,\nu)$ is strictly Henselian. Let $E/F_5$ be an arbitrary prime-to-$3$ extension. There is a unique extension of $\nu$ to $E$ which we will again denote $\nu$ by abuse of notation. Note that $(E,\nu)$ is strictly Henselian \cite[Proposition A.30]{tignol2015value}. The torsor $S_E$ corresponds to the $H$-Galois field extension $E_5:= E(t_0^{1/3},\dots,t_4^{1/3}) = E\otimes_{F_5} F_5(t_0^{1/3},\dots,t_4^{1/3})$ of $E$. The inclusion $H\subset G_k$ extends to an inclusion $H\subset G_{\cO}$ where $k\subset \cO \subset E$ is the valuation ring of $\nu$. Moreover, $C_{G}(H) = H$ implies that $H$ is not contained in any proper parabolic of $G_{\overline{k}}$ by Lemma~\ref{lem.finite_centralizer}. Therefore $T_{3,E}= \iota_*(S_E)$ is anisotropic by Lemma~\ref{lem.aniso}. 
This shows that $T_3$ satisfies the condition of Proposition~\ref{prop.ed_3_E8} and therefore $\ed_k(R_{G}(T_3);3) \geqslant 5$. 

\smallskip
We now explain how to generalize to the case where $k$ is not separably closed. We will focus on part (iii), which is the most complicated.  Parts (i) and (ii) are handled in an analogous way.  Denote $F'_5 = k^{\sep}((t_0))\dots((t_{4}))$, $K_5 = k(t_0,\dots,t_4)$ and let $G= E_8$.  The above proof gives us a $G$-torsor $T'_3$ over $F'_{5}$  whose Rost invariant satisfies $\ed_{k^{\sep}}(R_{G}(T'_3,3)) = 5$. To check that $T'_5$ descends to $K_5$, we note that the embedding $\iota: H \hookrightarrow G$ is defined over $k$ because $k$ contains a primitive $3$-rd root of unity. We have $T'_3 = \iota_*(S)_{F'_5}$, where $S \in H^1(K_5,H)$ is the torsor corresponding to the $H$-Galois extension $L_5 = K_5(t_0^{1/3},\dots,t_4^{1/3})$ of $K_5$. We set $T_3 = \iota_*(S)_{F_5}$. The inequality $\ed_k(R_G(T_3);3)\leqslant 5$ follows from \eqref{e.morphism} because $\ed_k(T_3) \leqslant \trdeg_k(K_5) = 5$. Moreover, $T_{3,F'_5} = T'_3$ and so \eqref{e.ext_of_scalar} gives us:
$$\ed_k(R_G(T_3);3) \geqslant \ed_k(R_G(T_3)_{F'_5};3) = \ed_{k^{\sep}}(R_G(T'_3);3) = 5.$$
We conclude that $\ed_k(R_G(T_3);3) = 5$.
\end{proof}

\begin{remark} \label{rem.spin} 
One may wonder if similar techniques can be used to recover the exponential lower bounds 
\begin{equation} \label{e.brv}
\ed(\Spin_n; 2) \geqslant \begin{cases} 2^{(n-1)/2} - \frac{(n-1)n}{2} & \text{if $n$ is odd, and} \\
2^{(n-2)/2} - \frac{(n-1)n}{2} & \text{if $n$ is even}  \end{cases}
\end{equation}
from~\cite[Theorem 3.3]{brv-annals}. The answer is ``no".
(Here we assume that $k$ is of characteristic $0$ and $n \geqslant 15$. In the case, where $n$ is divisible by $4$, the lower bound in~\eqref{e.brv} is not optimal, but that shall not concern us here.) 

Indeed, recall that the Rost invariant for $R_G \colon H^1(F, \Spin_n) \to H^3(F, \mu_2)\subset H^3(F, \bQ/\bZ(2))$ for $G = \Spin_n$ may be computed as follows.  
Let $T \in H^1(F, \Spin_n)$ be a torsor. The image of $T$ under the natural map $H^1(F, \Spin_n) \to H^1(F, \SO_n)$ is represented by a quadratic form $q$ of rank $n$, with trivial discriminant and trivial Hasse-Witt invariant. 
Then $R_G(T) = {\rm ar}(q)$ is the Arason invariant of $q$; see \cite[p.~437]{knus1998involutions}.
In other words, if we represent $q$ as a sum of 3-fold Pfister forms 
\begin{equation} \label{e.3-pfister}
q = \langle\langle a_1, a_2, a_3 \rangle \rangle + \ldots +  \langle\langle a_{3r - 2} , a_{3r-1},  a_{3r} \rangle \rangle
\end{equation}
in the Witt ring $W(F)$, then 
\[ {\rm ar}(q) = (a_1)_2 \cup (a_2)_2 \cup (a_3)_2 + \ldots + (a_{3r-2})_2 \cup (a_{3r-1})_2 \cup (a_{3r})_2 \in H^3(K, \mu_2). \]
Note that $q$ can be represented in the form~\eqref{e.3-pfister} by a theorem of Merkurjev~\cite{merkurjev-I^3}.
% (this result was a precursor to and is a special case of the Merkurjev-Suslin Theorem).

Now assume $F$ is equipped with a strictly Henselian valuation.
By a theorem of Raczek~\cite[Theorem 1.13]{raczek}, we can assume $r \leqslant n^2/8$ in~\eqref{e.3-pfister}.
Since $R_G(T)= {\rm ar}(q)$ descends to $k(a_1,\dots,a_{3r})$ we conclude:
\[ \ed_k(R_G(T);2) \leqslant \ed_k(R_G(T)) \leqslant 3r \leqslant 3n^2/8. \]
For large $n$ this is smaller than the lower bound on $\ed_k(\Spin_n;2)$ in~\eqref{e.brv}.
\end{remark}
\begin{comment}
    By Theorem~\ref{thm.main}, $\ed_k(R_G(T);2) = \ed_k({\rm ar}(q);2) = \dim_{\bF_2} A_\omega/2$, where
\[ \omega = \wedge \nu ({\rm ar}(q)) = \frac{1}{2} \otimes \big( 
\nu(a_1) \wedge \nu(a_2) \wedge \nu(a_3) + \ldots + \nu(a_{3r-2}) \wedge \nu(a_{3r-1}) \wedge \nu(a_{3r}) \big) \]
is in $\wedge^3 \bZ^n / 2 \subset \wedge^3 \bZ^n / 2^{\infty}$. 
By a theorem of Raczek~\cite[Theorem 1.13]{raczek}, $\omega$ can be written 
as a sum of $m \leqslant n^2/8$ pure wedges, i.e.,
\[ w = v_1 \wedge v_2 \wedge v_3 + \ldots + v_{3m-2} \wedge v_{3m-1} \wedge v_{3m}  \]
for some $v_1, \ldots, v_{3m}$ in $\bZ^n/2$.
We conclude that $A_\omega$ is contained in the span of $v_1, \ldots v_{3m}$ and thus
\[ \ed_k(R_G(T);2) = \dim_{\bF_2} A_\omega/2 \leqslant 3m \leqslant 3n^2/8. \]
For large $n$ this is smaller than the lower bound on $\ed_k(\Spin_n;2)$ in~\eqref{e.brv}.
\end{comment}
\section{Appendix}

\subsection{The norm residue isomorphism}
Let $F$ be a field. Denote the separable closure of $F$ by $F^{\sep}$.
% and $m$ a positive integer coprime to $\Char F$.
Choose compatible a compatible system of roots of primitive $m$th unity $\zeta_m \in F^{\sep}$, as $m$ ranges over the positive integers prime to $\Char(F)$. That is, $\zeta_{m_1 m_2}^{m_1} = \zeta_{m_2}$, 
as in Section~\ref{sect.notational-conventions}.

Now fix one particular integer $m$ prime to $\Char(F)$.
Recall that the Kummer map is the connecting homomorphism :
\begin{equation}\label{e.kummer}
    \partial_m: F^* \to H^1(F,\mu_{m})
\end{equation}
induced by the short exact sequence $1 \to \mu_{m} \to (F^{\sep })^* \overset{\cdot m}{\to }(F^{\sep})^* \to 1$  \cite[Chapter 2, Section 1.2]{serre1997galois}. The norm residue isomorphism theorem states that for all $m$ coprime to $\Char k$ and $d\in \bN$ there is an isomorphism:
$$K_d(F)/m \to H^d(F,\mu^{\otimes d}_{m}), \ \ \{a_1,\dots,a_d\} \mapsto \partial_m(a_1)\cup \dots\cup \partial_m(a_d).$$
Here $K_d$ is Milnor $K$-theory. We will denote by $h^d_{m} : K_d(F) \to H^d(F,\bZ/{m}(d))$ the Galois symbol,
given by the composition: 
$$K_d(F)\to H^d(F,\mu^{\otimes d}_{m}) \to H^d(F,(\bZ/m(1))^{\otimes d}) \to   H^d(F,\bZ/m (d)).$$
For $d =0$, we set  $h^0_m : K_0(F) = \bZ \to \bZ/m$ to be the reduction modulo $m$ homomorphism.  Note that $h^d_{m}$ depends on our choice of the root of unity $\zeta_m$, which induces an isomorphism $\mu_{m}\to\bZ/m(1)$. 

We will use the following notation for the symbol defined by $a_1,\dots,a_d\in F^*$:
$$(a_1,\dots,a_d)_{m}:= h^d_{m} (\{a_1,\dots,a_d\}).$$
Putting these together for all $d \geqslant 0$ yields a homomorphism of graded groups
$$h_m:= \oplus h^d_m : K(F) \overset{}{\to}\bigoplus_d H^d(F,\bZ/m (d)).$$ 
As we vary $m$, the homomorphisms $h_m$ transform like our chosen roots of unity $\zeta_m$.
\begin{lemma}\label{lem.pow}
For all $m \mid n$ coprime to $\Char F$ we have:
$$i_{m,n} \circ h_{m} = \frac{n}{m} h_{n}$$
where $i_{m,n} : \bigoplus_d H^d(F,\bZ/m (d))\to \bigoplus_d H^d(F,\bZ/n (d))$ is the natural inclusion.
\end{lemma}
\begin{proof}
Let $\iota: \bZ/m (d) \to \bZ/n (d)$ be the inclusion given by multiplication by ${\frac{n}{m}}$ and let $\iota_*: H^d(F,\bZ/m (d))\to H^d(F,\bZ/n (d))$ be the push-forward on Galois cohomology. It suffices to prove for all $a_1,\dots,a_d \in F^*$:
\begin{equation}\label{e.symb}
    \iota_*(a_1,\dots,a_d)_{m} = \frac{n}{m}(a_1,\dots,a_d)_{n}.
\end{equation}
We do this on the level of cocycles. Let $b_i\in {F^{\sep}}$ be an $n$-th root of $a_i$. By definition of the Kummer map, the function $$c'_{i}(\sigma) := b_i \sigma(b_i)^{-1}\in \mu_{n}$$ is a cocycle representing the class of $\partial_{n}(a_i)\in H^1(F,\mu_{n})$; see \eqref{e.kummer}. Similarly, the cocycle $$c_i(\sigma) = b_i^{\frac{n}{m}}\sigma(b_i^{-\frac{n}{m}}) = (b_i\sigma(b_i^{-1}))^{\frac{n}{m}}=c'_i(\sigma)^{\frac{n}{m}}$$ represents the class of $\partial_{m}(a_i)\in H^1(F,\mu_{m})$. Using our choice of isomorphisms $\mu_n \cong \bZ/n(1)$, we can find a cocycle $\chi_i: \Gal(F) \to \bZ/n(1)$ such that 
\begin{equation}\label{e.cocycle}
    \text{for all }\sigma\in \Gal(F): c'_{i}(\sigma) = \zeta_{n}^{\chi_i(\sigma)}
\end{equation}
We calculate:
$$c_i(\sigma) = c'_i(\sigma)^{\frac{n}{m}}= (\zeta^{\chi_i(\sigma)}_{n})^{\frac{n}{m}} = \zeta^{\chi_i(\sigma)}_m .$$
This shows that under the isomorphism $H^1(F,\mu_m) \tilde{\to} H^1(F,\bZ/m(1))$ 
the cohomology class $\partial_m(a_i)$ is sent to the class represented by the cocycle:
$$\Gal(F) \to \bZ/m(1),\ \ \sigma \mapsto \chi_i(\sigma) \mod m.$$
Therefore, $(a_1,\dots,a_d)_{m}$ is represented by the $n$-cocycle:
$$f:\Gal(F)^n \to \bZ/m (n),\ \ f(\sigma_1,\dots,\sigma_d) = \prod_{1\leqslant  i\leqslant  n}\chi_{i}(\sigma_i) \mod m.$$
Pushing $f$ forward using $\iota$, we get a cocycle representing $\iota_*(a_1,\dots,a_d)_{m}$:
$$\iota_*(f) (\sigma_1,\dots,\sigma_d) = \frac{n}{m} \prod_{1\leqslant  i\leqslant  n} \chi_i(\sigma_i) \mod n.$$
This cocycle represents $\displaystyle \frac{n}{m}(a_1,\dots,a_d)_n$ because the cocycle $\chi_i$ represents the image of $\partial_n(a_i)$ under the isomorphism $H^1(F,\mu_n) \tilde{\to} H^1(F,\bZ/n(1))$ by \eqref{e.cocycle}. This implies \eqref{e.symb} and finishes the proof.
\end{proof}

We can now deduce the "global" form of the norm residue isomorphism stated in Theorem~\ref{thm.norm} from the norm residue isomorphism theorem and Lemma~\ref{lem.pow}. Recall that $K(F)/ p^\infty = \bQ_p/\bZ_p \otimes_\bZ K(F)$ by definition.

\begin{theorem}
For any prime $p$ different from $\Char k$, there is an isomorphism of abelian graded groups $K(F)/p^{\infty} \overset{h}{\to} H_p(F)$ given on generators by:
$$h(\frac{1}{p^n}\otimes \{a_1,\dots,a_d\}) = h^d_{p^n}(\{a_1,\dots,a_d\}) = (a_1,\dots,a_n)_{p^n}$$
\end{theorem}
\begin{proof}
The homomorphism $h$ is well-defined by Lemma~\ref{lem.pow} and because $h_1 = 0$ by definition. Moreover $h$ is surjective by the norm residue isomorphism theorem. Assume $h(\alpha)=0$ for some $\alpha \in\bT K(F)$. We can write $\alpha \cong \frac{1}{p^n}x$ for some $x\in K(F)$ and $n\in \bN$. Since $h(\alpha) = h_{p^n}(x)=0$, the norm residue isomorphism theorem implies $x = p^n x'$ for some $x'\in K(F)$. Therefore $\alpha = 0$. This shows $h$ is indeed an isomorphism.
\end{proof}

\subsection{Extensions of valuations}

The following lemmas will be used to handle extensions of degree prime to $p$.

\begin{lemma}\label{lem.prime_to_p1}
    Let $(F,\nu)$ be a valued field over $k$ and let $F\subset L$ be a finite extension of degree prime to $p$. There exists an extension $w$ of $\nu$ to $L$ such that the ramification degree $[w L : \nu F]$ is prime to $p$.
\end{lemma}
\begin{proof}
    Let $w_1,\dots,w_r$ be all the extensions of $\nu$ to $L$.  Denote by $F^h$ the Henselization of $(F,\nu)$, and choose Henselizations  $F^h \subset L^h_i$ of $(L,w_i)$ (these exist by \cite[Theorem A.30]{tignol2015value}). Using \cite[Theorem A.32]{tignol2015value}, we get an $F^h$-linear isomorphism:
    $$L \otimes_F F^h \cong L^h_1 \times \dots \times L^h_r.$$
    Comparing dimensions over $F^h$ we find:
    \begin{equation}\label{eq.deg_henselizations}
        [L:F] = [L^h_1: F^h] + \dots +[L^h_r: F^h].
    \end{equation}
    Let $e_i = [w_i L:\nu F]$ be the ramification degree of $(L,w_i)$ over $(F,\nu)$ and note that it is also the ramification degree of $(L^h_i, w_i)$ over $(F^h, \nu)$. Ostrowski's theorem implies $e_i$ divides $[L^h_i: F^h]$ for all $i$ \cite[Theorem A.12]{tignol2015value}. Since the left hand side of \eqref{eq.deg_henselizations} is prime to $p$, we conclude that $e_i$ is prime to $p$ for some $i$.   
\end{proof}

\begin{lemma}\label{lem.prime_to_p2}
    Let $\Gamma \subset \Gamma'$ be finitely generated free abelian groups such that $[\Gamma':\Gamma]$ is prime to $p$. The canonical inclusion $\bigwedge \Gamma \subset \bigwedge \Gamma'$ induces an isomorphism:
    \begin{equation} \label{e.prime_to_p2}
    \bigwedge \Gamma / p^\infty \cong \bigwedge \Gamma'/ p^\infty .
    \end{equation}
\end{lemma}
   
    \begin{proof}
        % Using the Smith normal form 
        By the Elementary Divisors Theorem (as in the proof of Proposition~\ref{prop.ed_limit}),
      %   we can assume without loss of generality that
      $\Gamma'$ has a basis $e_1,\dots,e_r$ such that $d_1 e_1,\dots,d_r e_r$ is a basis of $\Gamma$ for some positive integers $d_1, \ldots, d_r$.
      Then $[\Gamma':\Gamma] = \prod d_i$; in particular, $d_i$ is prime to $p$ for every $i$. The group $\bigwedge \Gamma'$ has a basis $\{e_{i_1}\wedge\dots\wedge e_{i_k}\}$ indexed by increasing sequences $1\leqslant  i_1 < \dots<i_k \leqslant  r$. Similarly, the pure tensors $\{d_{i_1}e_{i_1}\wedge\dots\wedge d_{i_k}e_{i_k}\}$ form a basis for $\bigwedge \Gamma$. Therefore the index $[\bigwedge \Gamma' : \bigwedge \Gamma]$ is prime to $p$.
      Consequently, the inclusion $\bigwedge \Gamma \hookrightarrow \bigwedge \Gamma'$ induces an isomorphism
      $(\bigwedge \Gamma)/p^n \stackrel{\sim}{\to} (\bigwedge \Gamma') /p^n$ for every $n \geqslant 1$. 
      Taking the colimit as $n \to \infty$,
      we obtain the isomorphism ~\eqref{e.prime_to_p2}. 
    \end{proof}

%Intro to appendix

%%%%%%%%%%%%%%%%%%%%%%%%%%%%%%%%%%%%%%%%%%%%%%%%%%%%%%%%%

\bibliographystyle{plain}
\bibliography{references.bib}

\begin{thebibliography}{10}

\bibitem{berhuy2003essential}
Gr{\'e}gory Berhuy and Giordano Favi.
\newblock Essential dimension: a functorial point of view (after {A}.
  {M}erkurjev).
\newblock {\em Doc. Math. 2003;8(106):279-330.}, 2003.

\bibitem{borel_lag}
Armand Borel.
\newblock {\em Linear algebraic groups}, volume 126 of {\em Graduate Texts in
  Mathematics}.
\newblock Springer-Verlag, New York, second edition, 1991.

\bibitem{bourbaki_ac_1-7}
Nicolas Bourbaki.
\newblock {\em Commutative algebra. {C}hapters 1--7}.
\newblock Elements of Mathematics (Berlin). Springer-Verlag, Berlin, 1998.
\newblock Translated from the French, Reprint of the 1989 English translation.

\bibitem{brv-annals}
Patrick Brosnan, Zinovy Reichstein, and Angelo Vistoli.
\newblock Essential dimension, spinor groups, and quadratic forms.
\newblock {\em Ann. of Math. (2)}, 171(1):533--544, 2010.

\bibitem{brussel2001division}
E.~S. Brussel.
\newblock The division algebras and {B}rauer group of a strictly {H}enselian
  field.
\newblock {\em J. Algebra}, 239(1):391--411, 2001.

\bibitem{chernousov2003kernel}
V.~Chernousov.
\newblock The kernel of the {R}ost invariant, {S}erre's conjecture {II} and the
  {H}asse principle for quasi-split groups {${}^{3,6}D_4,E_6,E_7$}.
\newblock {\em Math. Ann.}, 326(2):297--330, 2003.

\bibitem{chernousov2006lower}
Vladimir Chernousov and Jean-Pierre Serre.
\newblock Lower bounds for essential dimensions via orthogonal representations.
\newblock {\em Journal of Algebra}, 305(2):1055--1070, 2006.

\bibitem{elman1982arason}
Richard Elman.
\newblock On arason's theory of galois cohomology.
\newblock {\em Communications in Algebra}, 10(13):1449--1474, 1982.

\bibitem{engler2005valued}
Antonio~J Engler and Alexander Prestel.
\newblock {\em Valued fields}.
\newblock Springer Science \& Business Media, 2005.

\bibitem{garibaldi2001rost}
Ryan~Skip Garibaldi.
\newblock The rost invariant has trivial kernel for quasi-split groups of low
  rank.
\newblock {\em Commentarii Mathematici Helvetici}, 76(4):684--711, 2001.

\bibitem{garibaldi2016shells}
S.~Garibaldi, V.~Petrov, and N.~Semenov.
\newblock Shells of twisted flag varieties and the {R}ost invariant.
\newblock {\em Duke Math. J.}, 165(2):285--339, 2016.

\bibitem{garibaldi2009orthogonal}
Skip Garibaldi.
\newblock Orthogonal involutions on algebras of degree 16 and the {K}illing
  form of {$E_8$}.
\newblock In {\em Quadratic forms---algebra, arithmetic, and geometry}, volume
  493 of {\em Contemp. Math.}, pages 131--162. Amer. Math. Soc., Providence,
  RI, 2009.
\newblock With an appendix by Kirill Zainoulline.

\bibitem{garibaldi2003cohomological}
Skip Garibaldi, Alexander Merkurjev, and Jean-Pierre Serre.
\newblock {\em Cohomological invariants in {G}alois cohomology}, volume~28 of
  {\em University Lecture Series}.
\newblock American Mathematical Society, Providence, RI, 2003.

\bibitem{gille2009lower}
Philippe Gille and Zinovy Reichstein.
\newblock A lower bound on the essential dimension of a connected linear group.
\newblock {\em Commentarii Mathematici Helvetici}, 84(1):189--212, 2009.

\bibitem{gille2017central}
Philippe Gille and Tam\'{a}s Szamuely.
\newblock {\em Central simple algebras and {G}alois cohomology}, volume 165 of
  {\em Cambridge Studies in Advanced Mathematics}.
\newblock Cambridge University Press, Cambridge, 2017.
\newblock Second edition of [ MR2266528].

\bibitem{griess1991elementary}
Robert~L Griess~Jr.
\newblock Elementary abelian p-subgroups of algebraic groups.
\newblock {\em Geometriae Dedicata}, 39(3):253--305, 1991.

\bibitem{hodge1954}
W.~V.~D. Hodge and D.~Pedoe.
\newblock {\em Methods of algebraic geometry. {V}ol. {III}}.
\newblock Cambridge Mathematical Library. Cambridge University Press,
  Cambridge, 1994.
\newblock Book V: Birational geometry, Reprint of the 1954 original.

\bibitem{jantzen2004nilpotent}
Jens~Carsten Jantzen, Karl-Hermann Neeb, and Jens~Carsten Jantzen.
\newblock Nilpotent orbits in representation theory.
\newblock {\em Lie theory: Lie algebras and representations}, pages 1--211,
  2004.

\bibitem{book_involution}
Max-Albert Knus, Alexander Merkurjev, Markus Rost, and Jean-Pierre Tignol.
\newblock {\em The book of involutions}, volume~44 of {\em American
  Mathematical Society Colloquium Publications}.
\newblock American Mathematical Society, Providence, RI, 1998.
\newblock With a preface in French by J. Tits.

\bibitem{knus1998involutions}
Max-Albert Knus, Alexander Merkurjev, Markus Rost, and Jean-Pierre Tignol.
\newblock {\em The book of involutions}, volume~44 of {\em American
  Mathematical Society Colloquium Publications}.
\newblock American Mathematical Society, Providence, RI, 1998.
\newblock With a preface in French by J. Tits.

\bibitem{lang2002algebra}
Serge Lang.
\newblock {\em Algebra}, volume 211 of {\em Graduate Texts in Mathematics}.
\newblock Springer-Verlag, New York, third edition, 2002.

\bibitem{lorenz2003}
M.~Lorenz, Z.~Reichstein, L.~H. Rowen, and D.~J. Saltman.
\newblock Fields of definition for division algebras.
\newblock {\em J. London Math. Soc. (2)}, 68(3):651--670, 2003.

\bibitem{mckinnie2017essential}
Kelly McKinnie.
\newblock Essential dimension of generic symbols in characteristic.
\newblock In {\em Forum of Mathematics, Sigma}, volume~5, page e14. Cambridge
  University Press, 2017.

\bibitem{merkurjev-I^3}
A.~S. Merkurjev.
\newblock On the norm residue symbol of degree {$2$}.
\newblock {\em Dokl. Akad. Nauk SSSR}, 261(3):542--547, 1981.

\bibitem{merkurjev-ci}
Alexander Merkurjev.
\newblock Rost invariants of simply connected algebraic groups.
\newblock In {\em Cohomological invariants in {G}alois cohomology}, volume~28
  of {\em Univ. Lecture Ser.}, pages 101--158. Amer. Math. Soc., Providence,
  RI, 2003.
\newblock With a section by Skip Garibaldi.

\bibitem{merkurjev2009essential}
Alexander~S. Merkurjev.
\newblock Essential dimension.
\newblock In {\em Quadratic forms---algebra, arithmetic, and geometry}, volume
  493 of {\em Contemp. Math.}, pages 299--325. Amer. Math. Soc., Providence,
  RI, 2009.

\bibitem{merkurjev-PGLp^2}
Alexander~S. Merkurjev.
\newblock Essential {$p$}-dimension of {${\rm PGL}(p^2)$}.
\newblock {\em J. Amer. Math. Soc.}, 23(3):693--712, 2010.

\bibitem{merkurjev-PGLn}
Alexander~S. Merkurjev.
\newblock A lower bound on the essential dimension of simple algebras.
\newblock {\em Algebra Number Theory}, 4(8):1055--1076, 2010.

\bibitem{merkurjev-survey}
Alexander~S. Merkurjev.
\newblock Essential dimension: a survey.
\newblock {\em Transform. Groups}, 18(2):415--481, 2013.

\bibitem{meyer2012valuation}
Aurel Meyer.
\newblock A valuation theoretic approach to essential dimension.
\newblock {\em arXiv preprint arXiv:1202.5363}, 2012.

\bibitem{milne2017algebraic}
J.~S. Milne.
\newblock {\em Algebraic groups}, volume 170 of {\em Cambridge Studies in
  Advanced Mathematics}.
\newblock Cambridge University Press, Cambridge, 2017.
\newblock The theory of group schemes of finite type over a field.

\bibitem{neukirch2013cohomology}
J{\"u}rgen Neukirch, Alexander Schmidt, and Kay Wingberg.
\newblock {\em Cohomology of number fields}, volume 323.
\newblock Springer Science \& Business Media, 2013.

\bibitem{newman1972integral}
Morris Newman.
\newblock {\em Integral matrices}.
\newblock Academic Press, 1972.

\bibitem{ofek2022reduction}
D.~Ofek.
\newblock Reduction of structure to parabolic subgroups.
\newblock {\em Documenta Mathematica}, 27:1421--1446, 2022.

\bibitem{procesi}
Claudio Procesi.
\newblock Non-commutative affine rings.
\newblock {\em Atti Accad. Naz. Lincei Mem. Cl. Sci. Fis. Mat. Natur. Sez. I
  (8)}, 8:237--255, 1967.

\bibitem{raczek}
M\'{e}lanie Raczek.
\newblock On the 3-{P}fister number of quadratic forms.
\newblock {\em Comm. Algebra}, 41(1):342--360, 2013.

\bibitem{reichstein2010essential}
Zinovy Reichstein.
\newblock Essential dimension.
\newblock {\em Proceedings of the International Congress of Mathematicians
  (Vol. 2)}, 2010.

\bibitem{reichstein-vistoli-prime}
Zinovy Reichstein and Angelo Vistoli.
\newblock Essential dimension of finite groups in prime characteristic.
\newblock {\em C. R. Math. Acad. Sci. Paris}, 356(5):463--467, 2018.

\bibitem{reichstein2000essential}
Zinovy Reichstein and Boris Youssin.
\newblock Essential dimensions of algebraic groups and a resolution theorem for
  g-varieties.
\newblock {\em Canadian Journal of Mathematics}, 52(5):1018--1056, 2000.

\bibitem{reichstein-youssin}
Zinovy Reichstein and Boris Youssin.
\newblock Essential dimensions of algebraic groups and a resolution theorem for
  {$G$}-varieties.
\newblock {\em Canad. J. Math.}, 52(5):1018--1056, 2000.
\newblock With an appendix by J\'anos Koll\'ar and Endre Szab\'o.

\bibitem{ruozzi}
Anthony Ruozzi.
\newblock Essential {$p$}-dimension of {${\rm PGL}_n$}.
\newblock {\em J. Algebra}, 328:488--494, 2011.

\bibitem{sansuc1981groupe}
J.-J. Sansuc.
\newblock Groupe de {B}rauer et arithm\'{e}tique des groupes alg\'{e}briques
  lin\'{e}aires sur un corps de nombres.
\newblock {\em J. Reine Angew. Math.}, 327:12--80, 1981.

\bibitem{serre1997galois}
Jean-Pierre Serre.
\newblock {\em Galois cohomology}.
\newblock Springer-Verlag, Berlin, 1997.
\newblock Translated from the French by Patrick Ion and revised by the author.

\bibitem{serre-ci}
Jean-Pierre Serre.
\newblock Cohomological invariants, {W}itt invariants, and trace forms.
\newblock In {\em Cohomological invariants in {G}alois cohomology}, volume~28
  of {\em Univ. Lecture Ser.}, pages 1--100. Amer. Math. Soc., Providence, RI,
  2003.
\newblock Notes by Skip Garibaldi.

\bibitem{springer-steinberg}
T.~A. Springer and R.~Steinberg.
\newblock Conjugacy classes.
\newblock In {\em Seminar on {A}lgebraic {G}roups and {R}elated {F}inite
  {G}roups ({T}he {I}nstitute for {A}dvanced {S}tudy, {P}rinceton, {N}.{J}.,
  1968/69)}, Lecture Notes in Mathematics, Vol. 131, pages 167--266. Springer,
  Berlin, 1970.

\bibitem{tignol2015value}
Jean-Pierre Tignol and Adrian~R Wadsworth.
\newblock {\em Value functions on simple algebras, and associated graded
  rings}, volume~6.
\newblock Springer, 2015.

\bibitem{tits1965classification}
Jacques Tits.
\newblock Classification of algebraic semisimple groups.
\newblock In {\em Algebraic Groups and Discontinuous Subgroups (Proc. Sympos.
  Pure Math., Boulder, Colo., 1965)}, volume~9, pages 33--62, 1965.

\bibitem{wadsworth1983henselian}
Adrian~R. Wadsworth.
\newblock {$p$}-{H}enselian field: {$K$}-theory, {G}alois cohomology, and
  graded {W}itt rings.
\newblock {\em Pacific J. Math.}, 105(2):473--496, 1983.

\bibitem{zariski1960commutative}
Oscar Zariski and Pierre Samuel.
\newblock {\em Commutative algebra. {V}ol. {II}}.
\newblock The University Series in Higher Mathematics. D. Van Nostrand Co.,
  Inc., Princeton, N. J.-Toronto-London-New York, 1960.

\end{thebibliography}

\end{document}